\documentclass[reqno]{amsart}

\usepackage[all]{xy}

\usepackage{amssymb}
\usepackage{amsmath}
\usepackage{mathrsfs}
\usepackage{epsfig}
\usepackage{amscd}
\usepackage{graphicx}
\usepackage{color}
\usepackage{geometry}
\usepackage{hyperref}

\usepackage{tikz-cd}
\usepackage{tikz}
\usetikzlibrary{matrix,arrows,decorations.pathmorphing}
	\usepackage{amsthm, latexsym, mathtools}
    \usepackage{bbm}

\def\E{\ifmmode{\mathbb E}\else{$\mathbb E$}\fi} 
\def\N{\ifmmode{\mathbb N}\else{$\mathbb N$}\fi} 
\def\R{\ifmmode{\mathbb R}\else{$\mathbb R$}\fi} 
\def\Q{\ifmmode{\mathbb Q}\else{$\mathbb Q$}\fi} 
\def\C{\ifmmode{\mathbb C}\else{$\mathbb C$}\fi} 
\def\H{\ifmmode{\mathbb H}\else{$\mathbb H$}\fi} 
\def\Z{\ifmmode{\mathbb Z}\else{$\mathbb Z$}\fi} 
\def\P{\ifmmode{\mathbb P}\else{$\mathbb P$}\fi} 
\def\T{\ifmmode{\mathbb T}\else{$\mathbb T$}\fi} 
\def\SS{\ifmmode{\mathbb S}\else{$\mathbb S$}\fi} 
\def\DD{\ifmmode{\mathbb D}\else{$\mathbb D$}\fi} 

\newcommand{\del}{\partial}

\newcommand{\ben}{\begin{enumerate}}
\newcommand{\een}{\end{enumerate}}
\newcommand{\be}{\begin{equation}}
\newcommand{\ee}{\end{equation}}
\newcommand{\bea}{\begin{eqnarray}}
\newcommand{\eea}{\end{eqnarray}}
\newcommand{\beastar}{\begin{eqnarray*}}
\newcommand{\eeastar}{\end{eqnarray*}}
\newcommand{\bc}{\begin{center}}
\newcommand{\ec}{\end{center}}

\theoremstyle{plain} 
\newtheorem{thm}{Theorem}[section]
\newtheorem{cor}[thm]{Corollary}
\newtheorem{lem}[thm]{Lemma}
\newtheorem{prop}[thm]{Proposition}

\newtheorem{sublem}[thm]{Sublemma}

\theoremstyle{definition}
\newtheorem{exm}[thm]{Example}

\newtheorem{defn}[thm]{Definition}
\newtheorem{rem}[thm]{Remark}
\newtheorem{ques}[thm]{Question}

\newtheorem{cond}[thm]{Condition}

\newtheorem*{thm*}{Theorem}

\numberwithin{equation}{section}

\def\R{{\mathbb R}}

\def\E{{\mathbb E}}
\def\Z{{\mathbb Z}}
\def\C{{\mathbb C}}
\def\R{{\mathbb R}}
\def\P{{\mathbb P}}

\def\N{{\mathbb N}}

\def\11{{\mathbb I}}

\def\C{\mathbb{C}}
\def\Z{\mathbb{Z}}

\def\T{\mathbb{T}}

\def\Q{\mathbb{Q}}
\def\Y{\mathbb{Y}}

\def\E{\ifmmode{\mathbb E}\else{$\mathbb E$}\fi} 
\def\N{\ifmmode{\mathbb N}\else{$\mathbb N$}\fi} 
\def\R{\ifmmode{\mathbb R}\else{$\mathbb R$}\fi} 
\def\Q{\ifmmode{\mathbb Q}\else{$\mathbb Q$}\fi} 
\def\C{\ifmmode{\mathbb C}\else{$\mathbb C$}\fi} 
\def\H{\ifmmode{\mathbb H}\else{$\mathbb H$}\fi} 
\def\Z{\ifmmode{\mathbb Z}\else{$\mathbb Z$}\fi} 
\def\P{\ifmmode{\mathbb P}\else{$\mathbb P$}\fi} 
\def\SS{\ifmmode{\mathbb S}\else{$\mathbb S$}\fi} 
\def\DD{\ifmmode{\mathbb D}\else{$\mathbb D$}\fi} 

\def\R{{\mathbb R}}

\def\E{{\mathbb E}}
\def\Z{{\mathbb Z}}
\def\C{{\mathbb C}}
\def\R{{\mathbb R}}
\def\M{{\mathbb M}}

\def\N{{\mathbb N}}
\def\FF{{\mathcal F}}

\def\id{{\operatorname{id}}}

\def\CA{{\mathcal A}}

\def\CC{{\mathcal C}}

\def\CF{{\mathcal F}}

\def\CH{{\mathcal H}}

\def\CL{{\mathcal L}}
\def\CM{{\mathcal M}}
\def\CN{{\mathcal N}}

\def\CP{{\mathcal P}}

\def\CP{{\mathcal P}}
\def\CS{{\mathcal S}}

\def\CU{{\mathcal U}}
\def\CV{{\mathcal V}}

\def\darr#1{\raise1.5ex\hbox{$\leftrightarrow$}
\mkern-16.5mu #1}

\def\roughly#1{\raise.3ex\hbox{$#1$\kern-.75em
\lower1ex\hbox{$\sim$}}}

\def\opname#1{\mathop{\kern0pt{\rm #1}}\nolimits}

\def\Im{\opname{Im}}

\def\dim{\opname{dim}}
\def\vol{\opname{vol}}

\def\supp{\operatorname{supp}}

\def\span{\operatorname{span}}

\def\rank{\opname{rank}}
{

\def\supp{\operatorname{supp}}

\def\span{\operatorname{span}}

\def\Image{\operatorname{Image}}
\def\ev{\operatorname{ev}}

\def\codim{\operatorname{codim}}

\def\nullity{\operatorname{nullity}}
\def\image{\operatorname{Image}}

\def\nullity{\text{\rm nullity }}

\DeclareFontFamily{U}{MnSymbolC}{}
\DeclareSymbolFont{MnSyC}{U}{MnSymbolC}{m}{n}
\DeclareFontShape{U}{MnSymbolC}{m}{n}{
    <-6>  MnSymbolC5
   <6-7>  MnSymbolC6
   <7-8>  MnSymbolC7
   <8-9>  MnSymbolC8
   <9-10> MnSymbolC9
  <10-12> MnSymbolC10
  <12->   MnSymbolC12}{}
\DeclareMathSymbol{\intprod}{\mathbin}{MnSyC}{'270}

\begin{document}

\title[Stratifications associated to closed two-forms]
{Stratifications associated to generic closed two-forms and stratified $L_\infty$ spaces}
\author{Taesu Kim, Yong-Geun Oh}
\address{Department of Mathematics, POSTECH,
77 Cheongam-ro, Nam-gu, Pohang-si, Korea 37673}
\email{kimtaesu@postech.ac.kr}
\address{Center for Geometry and Physics, Institute for Basic Science (IBS), 79 Jigok-ro 127beon-gil, 
Nam-gu, Pohang-si, 
Gyeongsangbuk-do, KOREA 37673 \& POSTECH, Pohang-si, Korea 37673}
\email{yongoh1@postech.ac.kr}

\thanks{This work is supported by the IBS project \# IBS-R003-D1.}


\begin{abstract} Park and the second-named author \cite{oh-park:coisotropic} introduce
the deformation problem of coisotropic submanifolds of a symplectic 
 manifold as the study of  Mauer-Cartan moduli problem of
 an $L_\infty$ algebra attached to the foliation de-Rham complex 
associated to the null foliation of the corresponding presymplectic structure. The main purpose of
the present paper is to extend this study of $L_\infty$ structures to the case of generic closed two-forms
on arbitrary smooth manifolds \emph{as a stratified $L_\infty$ space}. 
We first prove that there exists a residual subset of closed 2-forms, which
we denote by $Z^2_{\text{\rm reg}}(M) \subset Z^2(M)$, such that any element 
$\omega$ therefrom  admits a Whitney stratification each of whose strata
is a presymplectic manifold. We then associate an $L_\infty$ space to each stratum
(and to its tubular neighborhood) and \emph{glue}  the collection of $L_\infty$ spaces 
to a global stratified $L_\infty$ space by the coordinate atlas, which is
a collection consisting of $L_\infty$ morphisms,  \emph{not necessarily} of  quasi-isomorphisms. 
\end{abstract}


\maketitle

\tableofcontents

\section{Introduction and overview}

\subsection{Motivation}

The \emph{virtual intersection theory} on the moduli space of
pseudoholomorphic curves has played a fundamental role
in the development of  the Gromov-Witten theory both on the  
projective algebraic manifolds \cite{li-tian,behrend,behrend-fantechi} 
and on the symplectic manifolds
\cite{fukaya-ono:arnold,ruan:virtual,liu-tian,siebert}.
 The crux of the machinery is the existence of compactification
that admits the so called \emph{virtual fundamental cycles} on the moduli space of
(pseudo)holomorphic curves. The symplectic form or the algebraicity of 
the bulk essentially enters in the compactification of the moduli space.
Furthermore the latter seems to play an important role in the construction of algebraic virtual
fundamental class in algebraic geometry in the study of infinitesimal deformation theory of
its  intrinsic normal cones. (See  \cite{li-tian}, \cite{behrend}, \cite{behrend-fantechi},
\cite{graber-virtual}.)
Indeed in the recent development of virtual intersection theory
in relation to the moduli theory of Donaldson-Thomas invariants \cite{oh-thomas} and (higher genus) 
Gromov-Witten invariants, the concept of  \emph{shifted symplectic geometry} \cite{shifted}
has been useful to construct various types of
 virtual fundamental cycles or chains in algebraic geometry.

On the other hand, as far as the present authors are aware, the symplectic form
has \emph{not} played any role in the \emph{infinitesimal} deformation theory of the moduli space
of pseudoholomorphic curves \emph{in the smooth context}.  
As a consequence, the compactified moduli space does not
see the difference, in the infinitesimal level,
 of whether the bulk is just an almost complex $(M,J)$ or an almost K\"ahler $(M,\omega, J)$.
However it seems to be natural to us that even in the smooth context
one anticipate some role of the \emph{`closed 2-form'}
on the moduli space  of pseudoholomorphic curves induced from the symplectic form on the bulk. 
Some more detail on this line of thinking is now in order.

Let $(\Sigma,j)$ be a compact Riemann surface and equip it with the Poinca\'re metric denoted by $h = h(j)$.
We denote by $dA$ the associated volume form of $\Sigma$.
A natural geometric structure one can think of in the latter case is the fact that the space of
smooth maps 
$$
\CF_{\Sigma}(M,\omega) : = C^\infty(\Sigma, (M,\omega))
$$
carries a canonical $L^2$-symplectic form defined by
\be\label{eq:L2symplectic}
\widetilde \omega(u)(\xi_1,\xi_2): = \int_\Sigma \omega(\xi_1, \xi_2)\, dA
\ee
for sections $\xi_1, \, \xi_2 \in \Gamma(u^*TM)$.  This form is \emph{weakly nondegenerate}
on $\CF_{\Sigma}(M,\omega)$ and canonically induces a \emph{closed} two-form
on any smooth finite dimensional submanifold thereof. Such a natural form does not 
exist when the bulk is just an almost complex manifold $(M, J)$. One might attempt to
exploit this presence of canonical differential `closed'  two-form in the (virtual) deformation theory
of the moduli space. For this purpose, the first natural question to ask is the following:

\begin{ques}\label{ques:form}Let 
$$
\widetilde \CM((\Sigma,j, (M,J); \beta)) \subset \CF((\Sigma,j), (M,J,\omega); \beta))
$$
be the moduli space  of $(j,J)$-holomorphic maps in class $\beta \in H_2(M,\Z)$. 
Suppose that the moduli space
 is smooth. Would the closed two form $\widetilde \omega_\beta$ thereon induced from 
 $ \CF((\Sigma, h), (M,\omega); \beta))$ carry a good structure of \emph{presymplectic 
 stratified space} for a `generic' choice of $J$?
 \end{ques}
 As it is, the question is quite ill-formulated by several reasons: For example,
 \begin{enumerate}
 \item In general we cannot achieve the smooth hypothesis imposed by perturbation of $J$ alone.
 \item Without having the smoothness, the notion of differential form $\widetilde \omega_\beta$ does not make
 sense let alone its closedness.
 \end{enumerate}
 At least, both should be considered in the sense of \emph{virtual deformation theory}
 as in the context of derived geometry, for example. We postpone this virtual theory as a future work.
 We refer to \cite{taesu} for a preliminary attempt to involve the presymplectic structure in the
 definition of (enhanced) Kuranishi structures. See \cite{costello}, \cite{grady-gwiliam1,grady-gwiliam2},
 \cite{amorim-tu}, \cite{pingxu}, \cite{behrend-liao-xu} 
for the derived geometry in various contexts of smooth differential geometry.

Once the above question is answered, the following is the next natural question to ask.

\begin{ques}\label{ques:volume} Assume the above `closed' two-form is equipped with the 
moduli space (or compactified moduli space).
\begin{enumerate}
\item Is the closed two-form nondegenerate away form codimension 1 subset?
\item Let $(\Sigma_g,j)$ be a Riemann surface of genus $g$ equipped with the 
Poincar'e metric $h$ of $\text{\rm Area}(\Sigma,h) = 1$. Compute the Liouville volume of the moduli space 
$\CM_g((M,J); \beta))$ in terms of $(M,\omega)$, genus $g$ and $\beta$.
\end{enumerate}
\end{ques}
Recall that when the target $(M,\omega)$ is a point,  i.e., the case of Deligne-Mumford 
moduli space, the structure of the collection of the Weil-Peterson volumes of compact Riemann surfaces
with geodesic boundaries under the gluing
operations has been extensively studied by Mirzakhani \cite{mirzakhani} in relation to
the intersection theory of the moduli space of curves. 

\begin{rem}
In this regard, it would be interesting to answer the above questions even for 
the simplest case of Gromov-Witten moduli space of $\C P^n$ beyond the case of point.
In fact, it can be shown that when the target space $(M,\omega, J)$ is a K\"ahler manifold, 
not just an almost K\"ahler, the induced form is indeed nondegenerate on each 
smooth component of the Gromov-Witten moduli space. Therefore 
Question \ref{ques:volume} is even more well-posed and doable in that case. We hope to come back 
to this question elsewhere in the future.
\end{rem}

 \subsection{Statement of main results}
 
With the above motivation in our mind, we turn to the study of the problem of studying
the general geometry of  generic closed 2-forms on a finite dimensional  smooth manifold.
Since the cases of $N = \dim M  \leq 2$ are rather trivial, we will assume
\be\label{N>2}
N = \dim M \geq 3.
\ee
(There is also a good analytical reason to separate the cases of $\dim M \leq 2$ from the higher dimensional
cases in our consideration: We have nothing to study
for $\dim M =1$. On the other hand, in dimension 2 the behavior of 
the singularity of the Green function of a Laplacian is \emph{logarithmic} unlike the cases of $\dim M \geq 3$
which forces us to separate the case of $\dim =2$ from our consideration of transversality study in
the proof of Proposition \ref{prop:transversality}. The case of dimension 2 is, however, so simple that
it  can be easily analyzed by an explicit finite dimensional argument, 
without going through this infinite dimensional transversality theory.)
 
The main purpose of the present paper is to expose the infinitesimal structure of 
 a generic smooth closed two-form on a smooth finite dimensional manifold.
Let $M$ be a smooth manifold and $\omega$ be a closed 2-form
of \emph{not necessarily of constant rank}. When $\omega$ has constant rank
the pair $(M,\omega)$ is called a \emph{presymplectic manifold}. Gotay \cite{gotay}
proved any presymplectic manifold can be embedded into a symplectic manifold 
as a coisotropic submanifold and provided a normal form of the symplectic form
on a neighborhood of the associated coisotropic submanifold: The normal form involves
the associated null foliation of the closed two-form and a choice of 
transversal section of the leaves of foliations. Utilizing this normal form, Park and
the second-named author  \cite{oh-park:coisotropic}
described deformations of a coisotropic submanifold 
and its Maurer-Cartan moduli functor  in terms of an $L_\infty[1]$ 
algebras (or strongly homotopy Lie algebras). They called such a space
a \emph{strong homotopy Lie algebroid} which is commonly called an 
$L_\infty$ space in the literature. We also adopt the latter term in the present
paper.

We consider the following subsets of $M$
\be\label{eq:Mj}
Y_m(\omega) : = \{ x\in M \mid \nullity \omega_x = m\}
\ee
for each given integer $0 \leq m \leq N$.  We consider the decomposition of $M$ into
$$
M = \bigcup_{ 0 \leq m \leq N} Y_m(\omega).
$$
A priori the subsets $Y_m$ could be very wild in general.
The following is the first main result we establish. 

\begin{thm}[Theorem \ref{thm:adaptedness}] There exists a residual subset, denoted by $Z^2_{\text{\rm reg}}(M)$,
of $Z^2(M)$ such that the decomposition $\CS_\omega = \{Y_m(\omega)\}_{m=0}^{N}$
forms a Whitney stratification of $M$ for any $\omega \in Z^2_{\text{\rm reg}}(M)$.
\end{thm}
We call any such closed two-form a \emph{nice} closed two-form.
By definition, for any nice closed two-form $\omega$, we have a
Whitney stratification $\{Y_m(\omega)\}_m$ of $M$ such that $\omega$ has constant rank 
on each stratum of the stratification, i.e., each of them
is a \emph{presymplectic} manifold.

We now recall the following main result of \cite[Theorem 9.4]{oh-park:coisotropic}: 
For each given presymplectic manifold $(Y,\omega)$ and a decomposition
$TY = G \oplus T\CF$, we consider
the foliation de Rham complex of the foliation $\CF_\omega$ associated to the null distribution
$\ker \omega \subset TY$ of $(Y,\omega)$ which we denote by
$$
\Omega^*(\CF_\omega): = \Omega^*(\CF_\omega, d_{\CF_\omega}).
$$
Then we can canonically equip it with an $L_\infty$-structure on the graded complex
$$
\left( \bigoplus_* \Omega^*(\CF_\omega)[1], \{\mathfrak l_k^\omega\}_{k \geq 1}\right).
$$
Motivated by this, we now formally name the triple $(Y,\omega,G)$ as follows.
\begin{defn}[Polarized presymplectic manifolds] Let $(Y,\omega)$ be a presymplectic manifold and
$G \subset TY$ be a choice of complementary subbundle of $\CN \subset TY$ as above.
We call such a triple $(Y,\omega, G)$ a \emph{polarized presymplectic manifold}.
\end{defn}
Therefore the foliation de-Rham complex of each polarized presymplectic stratum 
$(Y_j, \iota_j^*\omega, G_j)$ canonically
carries the structure of $L_\infty[1]$ algebra structure.
(See also \cite{cattaneo-felder}, \cite{cattaneo-schatz}, \cite{le-oh:lcs}, \cite{LOTV} for 
different and more systematic presentations of the theorem in various generalized geometric contexts.)

However, the current situation deals with a new general 
situation different therefrom:  We consider two presymplectic manifolds 
$(Y_\alpha, \omega_\alpha)$, $(Y_{\alpha'},\omega_{\alpha'})$
of \emph{different nullities}  with the compatibility relation
$$
\ker \left(\pi_\alpha^{\alpha'}(\omega_\alpha) \right) \supset \ker \omega_{\alpha'}
$$
and their \emph{gluing problem} the outcome of which we regard as
a \emph{stratified $L_\infty$ space}, which we would like to regard as a special case of
stratified $L_\infty$ Kuranishi space arising from the closed form on a finite
dimensional manifold. (See Appendix B for
the relevant discussion in relation to the notion of \emph{$L_\infty$ Kuranishi structure}
formulated by the first named author in \cite{taesu} with the aim of applying it to the moduli problem.)

\begin{thm}[Theorem \ref{thm:gluing-morphisms}]
\label{thm:gluing-intro}
Let $\{(U_m\}_m$ be a Mather's compatible system of neighborhoods of $\{Y_m\}_m$
such that $U_m \cap U_{m'} = \emptyset$ whenever $\dim Y_m = \dim Y_{m'}$. We put
$$
\CV_\alpha^\omega = \{(U_\alpha, \pi_\alpha^*\omega_\alpha, \pi_\alpha^*G_\alpha)\}_\alpha,
\quad \omega_\alpha : = \iota_\alpha^*\omega
$$
which is the pull-back presymplectic atlas of the polarized presymplectic stratification
$$
\CP\CS{(M,\omega)}=\{(Y_\alpha, \omega_\alpha, G_\alpha)\}_{\alpha \in \mathfrak P}
$$
of $(M,\omega)$. Then the following holds:
\begin{enumerate}
\item For each consecutive pair $(\alpha,\alpha')$ with $\alpha < \alpha'$.
Then there exists an $L_\infty$ morphism 
\be\label{eq:Phi-alphaalpha'}
\mathfrak f^{\alpha\alpha'}: \mathfrak l^{\CV^\omega}_\alpha \to \mathfrak l^{\CV^\omega}_{\alpha'}.
\ee
\item For each triple $(\alpha,\beta,\gamma)$ with $\alpha < \beta < \gamma$, the composition 
$\mathfrak f^{\beta\gamma} \circ \mathfrak f^{\alpha\beta}$ are defined
on an open subset $V_{\alpha\beta\gamma} \subset V_\alpha \cap V_\beta \cap V_\gamma$.
\item The two $L_\infty$ morphisms 
$\mathfrak f^{\beta\gamma} \circ \mathfrak f^{\alpha\beta}$ and $\mathfrak f^{\alpha\gamma}$
are canonically $L_\infty$ pseudo-isotopic.
\end{enumerate}
\end{thm}
See Section \ref{sec:compatible-system-linftyspaces} for the definition of pseudo-isotopy.
We interpret this result as a gluing result of \emph{strongly homotopy Lie-algebroids} \cite{oh-park:coisotropic}
or the $L_\infty$ spaces of different (virtual) dimensions \cite{amorim-tu}.

\subsection{Discussion}

As far as the authors are aware, algebraic structures derived from geometry are
isomorphic under the (smooth) isotopy in a suitable category of 
algebraic structures  constructed, \emph{whenever relevant deformations 
are not singular}. In the \emph{regular} context,  relevant deformations are commonly \emph{reversible}
and the way how such an isomorphism 
is constructed follows 2 steps: First one defines the notion of 
homotopy between two morphisms $f_0$, $f_1$ and then establishes
a homotopy between the identity and the composition $f_1 \circ f_0$ (and $f_0\circ f_1$).
However there is a situation where the \emph{homotopy is directed} such as in Directed Algebraic Topology
\cite{grandis:directed1,grandis:directed2}. Our construction of an $L_\infty$ morphism, 
which is \emph{not necessarily}
a quasi-isomorphism, naturally belongs to the realm of this \emph{directed category}  because 
our $L_\infty$ morphism is constructed by a smooth deformation retraction that however
decreases the rank of the initial presymplectic form along the way.  Furthermore the 
smooth isotopy constructed involves an isotopy
that vanishes at infinity order in time \emph{at the moment where the rank drops}
 the construction of which cannot be reversed
in time. In fact, this isotopy should be regarded as a smooth map on the space-time
\emph{ in the  admissible category in the sense of \cite[Section 25.1]{fooo:book-kuranishi}}.

Because of these reasons, there is no means to construct a homotopy between
the identity map to the composition $f_0 \circ f_1$ or $f_0\circ f_1$ or both.
We refer readers to Section \ref{sec:gluing}.

An organization of the present paper is now in order. In Section \ref{sec:revisit},
we revisit the construction of the structure of an $L_\infty$ space
from \cite{oh-park:coisotropic} associated to each presymplectic manifold 
$(M,\omega)$. 
We slightly amplify the construction in relation to the notion of \emph{curved $L_\infty$ spaces}
\cite{costello}, \cite{amorim-tu} by describing an explicit family of coaugmentations, or $\mathfrak l_0$
terms of a curved $L_\infty$ space. After this, the paper is divided into 2 parts. In Part I, we
study the generic structure of the decomposition of $M$ into the subsets $Y_m$
on which the nullity of the  value $\omega_x$ at $x \in M$ is $m$ with $m = 1, \cdots, \dim M$.
We prove that this decomposition leads to a Whitney stratification so that each stratum $Y_m$
becomes a presymplectic manifold by  the restricted closed two-form $\omega|_{Y_m}$ thereon.
In Part II, we then prove compatibility of the $L_\infty$ structures associated to $Y_m$ 
by verifying the compatibility condition laid out in Theorem \ref{thm:gluing-intro}.

\medskip

{\bf Acknowledgement:} The current research was geminated during the  
IBS-CGP and MATRIX workshop on Symplectic Topology held in MATRIX 
Conference Center in 2022. We would like to acknowledge MATRIX and 
the Simons Foundation for their support and funding 
through the MATRIX-Simons Collaborative Fund. We also thank the local
organizer Brett Parker for his hard work and hospitality during the workshop.

\section{$L_\infty$ spaces and presymplectic manifolds:  Revisit}
\label{sec:revisit}

In this section, we recall the construction of the aforementioned $L_\infty$ structure associated
to a presymplectic manifold \cite{oh-park:coisotropic}. (See
\cite[Definition 2.3]{cattaneo-schatz}, \cite{LOTV}, for example.) Park and the second-named author
called it a \emph{strongly homotopy Lie algebroid} which is equivalent to the notion of
an \emph{$L_\infty$ space} in other literature.
 
\subsection{Curved $L_\infty$ spaces and tangent complexes}
\label{subsec:tangent-complexes}

In this subsection, we closely follow the exposition given by Amorim and Tu \cite{amorim-tu}
of the notion of \emph{curved} $L_\infty$ space introduced by Costello \cite{costello}. 

An $L_\infty$ space, denoted by $\mathbb M = (M,\mathfrak g)$ is given by a pair of $M$ a smooth
manifold and $\mathfrak g$ a curved $L_\infty$ algebra over the ring of smooth functions
$C^\infty(M)$, requiring that $\mathfrak g$ is of the form
$$
\mathfrak g = \mathfrak g_2 \oplus \cdots \oplus \mathfrak g_d
$$
where each $\mathfrak g_i$ is a vector bundle in degree $i$ on $M$ for some $d \geq 2$. After shifting by degree, this becomes
$$
\mathfrak g[1] = \mathfrak g[1]_1\oplus \cdots \oplus \mathfrak g[1]_{d-1}.
$$
Given an $L_\infty$ space $\M = (M,\mathfrak g)$, the curvature term $\mu_0 \in \mathfrak g_2$
is a section $\sigma$ of the bundle $\mathfrak g_2$. 
\emph{We equip $TY$ with a torsion free  flat connection
and $\mathfrak g$ with a flat connection on a neighborhood of the zero set $\mu_0^{-1}(0)$.}

At each zero point $p \in \mu_0^{-1}(0)$,  the 
tangent complex at $p$ is defined as the chain complex
\be\label{eq:tangent-complex}
T_{x_0} \M: = \left(T_{x_0} M  \stackrel{\nabla \mu_0|_{x_0}}{\longrightarrow} \mathfrak g_2|_{x_0} \stackrel{\mu_1|_{x_0}}{\longrightarrow} 
\mathfrak g_3|_{x_0} \stackrel{\mu_1|_{x_0}}{\longrightarrow} \mathfrak g_4|_{x_0}
\stackrel{\mu_1|_{x_0}}{\longrightarrow} \cdots\right)
\ee
This complex does not depend on the choice of aforementioned connection at the zero point $p \in \mu_0^{-1}(0)$.
In \cite{amorim-tu}, a homotopy theory of $L_\infty$ spaces is extensively developed to the level that
they proved a version of an inverse function theorem \cite[Theorem 1.1]{amorim-tu}.

\begin{defn}[Virtual dimension] Denote by $\chi(\mathfrak g_{x_0}, \mu_1|_{x_0})$ the Euler characteristic of
the complex $(\mathfrak g, \mu_1)$. Then we define the virtual dimension of $\M$ at $p$ to be
$$
\text{\rm vit. dim}_{x_0}: = \dim T_{x_0}M - \chi(\mathfrak g_{x_0}, \mu_1|_{x_0}).
$$
When $\text{\rm vit. dim}_{x_0}$ is constant over $p \in M$, then
we define $\text{\rm vir.dim} M$ to be the common dimension.
\end{defn}

\subsection{The case of presymplectic manifolds}

Let $(Y,\omega)$ be a presymplectic manifold and $\pi: E \rightarrow Y$ be 
the foliation cotangent bundle $E = T^*\CF_\omega \to Y$ of the null foliation of $\omega$.
We can equip a canonical symplectic structure on a neighborhood $U$ of the zero section
$o_E\cong Y$ of the form
$$
\widetilde \omega: = \pi^*\omega - d \theta_G
$$
by choosing a splitting
\be\label{eq:TM-splitting}
TY = G \oplus T\CF_\omega.
\ee
We denote by $\Pi_{\omega;G}$ the Poisson bivector field on $U$ associated to
the symplectic form $\widetilde \omega$.
Utilizing this decomposition, Park and the second-named author \cite{oh-park:coisotropic} 
associates an $L_\infty$ algebra 
with vanishing curvature $\mathfrak l_0 = 0$.

\begin{thm}[Theorem 9.4 \cite{oh-park:coisotropic}] 
\label{thm:oh-park-intro} For each given presymplectic manifold $(Y,\omega)$, we can 
canonically equip it with an $L_\infty$-structure on the graded complex
$$
\left( \bigoplus_* \Omega^*(\CF_\omega)[1], \{\mathfrak l_k^\omega\}_{k \geq 1}\right).
$$
We denote by $\mathfrak l^\omega[1]$ the corresponding  $L_\infty[1]$ algebra.
\end{thm}
 We now add a natural set of \emph{coaugmentations} or a
\emph{curvatures} $\mathfrak l_0$ via the canonical interior product
$$
\Gamma(TY) \to \Omega^1(Y); \quad X \mapsto \omega
$$
and the splitting \eqref{eq:TM-splitting}. The following lemma is obvious by
definition of presymplectic manifold $(Y,\omega)$.

\begin{lem} The interior product map $v \mapsto v \intprod \omega|_p$ defines 
an isomorphism 
$$
N_p\CF_\omega\to N_p^*\CF_\omega.
$$
\end{lem}
\begin{proof}  Note that for any $v \in T_pY$, we have $v \intprod \omega|_{T_p\CF_\omega} = 0$
by definition of $\CF_\omega$ so that $T_p\CF_\omega= \ker \omega|_p$. Therefore 
$v \intprod \omega|_p \in N_p^*\CF_\omega$. This linear map is an isomorphism 
because it is injective on $T_pY / \ker\omega|_p$ and $N_p \CF_\omega$ and $N_p^* \CF_\omega$
have the same dimension.
\end{proof}

Note that $X \mapsto X \intprod \omega$ naturally defines a foliation differential
one-form in $\Omega^1(\CF_\omega)$ by restricting it to $T\CF_\omega$.

\begin{prop} Let $X$ be a vector field on $Y$ and define the 1-form $\eta_X$ by
\be\label{eq:etaX}
\eta_X: = X \intprod \omega.
\ee
We now put 
$$
\mathfrak l_0^X: = \eta_X|_{T\CF_\omega} \in \Omega^1(\CF_\omega).
$$
Then  $d_{\CF_\omega}(\mathfrak l_0^X) = 0$. In particular the pair $\Y^X: = (Y, \{\mathfrak l_k\}_{k=0}^\infty)$ 
with $\mathfrak l_0 = \mathfrak l_0^X$ defines a
curved $L_\infty$ space.
\end{prop}
\begin{proof} We evaluate the differential form $\eta_X$ against $X_1, \, X_2 \in \Gamma(T\CF_\omega)$. 
We compute
$$
d\eta_X(X_1,X_2) = X_1[\eta_X(X_2)] - X_2[\eta_X(X_1)] - \eta_X([X_1,X_2]).
$$
Since $X_i \in \ker \omega$, $\eta_X(X_i) = \omega(X,X_i) = 0$. By the integrability of $\ker \omega$,
we also have $[X_1,X_2] \in \ker \omega$ and hence the last term also vanishes. 
This proves $d_{\CF_\omega}(\mathfrak l_0(X)) = 0$ and hence $\mathfrak l_1 \circ \mathfrak l_0 = 0$
which finishes the proof.
\end{proof}

For a general vector field $X$, 
we obtain a deformed $L_\infty$ space denoted by
 $$
\Y^{\omega;X} = (Y^{\omega;X}, \mathfrak l^{\omega;X})
 $$
 which has \emph{non-zero} curvature. At each point $p \in (\mathfrak l_0^X)^{-1}(0)$,
 we have the tangent complex at $p$
\be\label{eq:TpYY}
T_p\Y^{\omega;X} = T_pY \stackrel{\nabla \mathfrak l_0^X|_p}{\longrightarrow}
\left(\Omega^1(\CF_\omega) \stackrel{\mathfrak l_1^\omega}{\longrightarrow}
\Omega^2(\CF_\omega) \stackrel{\mathfrak l_1^\omega}{\longrightarrow} \Omega^3(\CF_\omega)
\stackrel{\mathfrak l_1^\omega}{\longrightarrow} \cdots \right)\Big|_p
\ee
with $\mathfrak l_1^\omega = d_{\CF_\omega}$, whose virtual dimension is given 
$$
\text{\rm vir.dim } T\Y^{\omega;X}|_p = \dim Y - \chi\left(H^*(\Omega^*(\CF_\omega), d_{\CF_\omega})\right), 
$$
Here $(\cdot)|_p$ is the stalk at $p$ of the de-Rham complex.

\begin{rem} The diagram \eqref{eq:TpYY} looks slightly different from the general abstract
definition of $L_\infty$ spaces given in \eqref{eq:tangent-complex}. The $L_\infty$ space relevant to 
this foliation de Rham complex is explained in \cite[Section 9]{oh-park:coisotropic} in the formal
level in the language of super-manifold through Gotay's coisotropic embedding of $(Y,\omega)$
into $U \subset T^*\CF_\omega$.
\end{rem}

\subsection{Formal deformation of coisotropic submanifolds}

We consider a section $\sigma : Y \to T^*\CF_\omega$ whose
image is contained in Gotay's neighborhood $U \subset T^*\CF_\omega$
ask what the condition for its image to be coisotropic with respect to 
the symplectic form $\widetilde \omega = \pi^*\omega - d\theta_G$ is.

\begin{defn}\label{defn:smcoiso}
A smooth one parameter family of smooth sections of $T^*\CF_\omega \to Y$ starting from the zero section is a \emph{smooth coisotropic deformation of} $S$ if each section in the family is coisotropic. A section $\sigma$ 
of $T^*\CF_\omega \to Y$ is an \emph{infinitesimal coisotropic deformation of} $S$ if $\epsilon s$ is a coisotropic section up to infinitesimals $O(\epsilon^2)$, where $\epsilon$ is a formal parameter.
\end{defn}

Recall that  a formal series $\sigma (\epsilon) = \sum _{k =0} ^ \infty \epsilon ^k \Gamma_k \in \Omega^1(\CF_\omega)
 [[ \epsilon]]$,
 $\Gamma_k \in  \Omega^1(\CF_\omega)$, such that $\Gamma_0 = 0$,
 is called a \emph{formal deformation} of coisotropic submanifold $Y = o_{T^*\CF_\omega}$.
 Theorem \ref{thm:oh-park} provides the criterion for 
$$
 \sigma(\epsilon) := \sum_{k=0}^\infty \epsilon^k \Gamma_k
 $$
 to be a formal coisotropic deformation where where $\Gamma_k$'s are sections of $T^*\CF_\omega$ and $\varepsilon$ is a formal parameter. The following formal definition is borrowed from 
 \cite[p.1080]{LOTV}.
 
 \begin{defn}[Maurer-Cartan series] Let $\sigma$ be a section of $T^*\CF_\omega$.
 The \emph{Maurer-Cartan series} of $\sigma$ is the series
 \be\label{eq:maurer-cartan}
 \operatorname{MC}(\sigma): = \sum_{k=0} \frac1{k!} \mathfrak l_K(\sigma,\cdots ,\sigma).
 \ee
 \end{defn}

Then the following description of Maurer-Cartan moduli space
is a translation of the result given  in \cite[Theorem 11.1]{oh-park:coisotropic}  for the deformation 
problem of coisotropic submanifolds modulo Hamiltonian equivalence.

\begin{thm}[Theorem 11.1 \cite{oh-park:coisotropic}] \label{thm:oh-park}
Let $\sigma \in \mathfrak{l}[1]^0 = \mathfrak{l}^1 = \Omega^1(\CF_\omega)$. Then 
$\image \sigma \subset U \subset T^*\CF_\omega$ is 
a coisotropic submanifold of  $\widetilde \omega$  (in the formal level)  if and only if
$$
\operatorname{MC}(\sigma) = 0
$$
for the formal power series
$\sigma = \sum_{k = 0}^\infty \epsilon^k \Gamma_k$.
\end{thm}
The statement of this theorem can be promoted to the \emph{fiberwise entire} cases
so that the above formal power series actually converges, 
as studied in  \cite{schatz-zambon:convergent}, \cite[Theorem 4.23 \& Corollary 4.24]{LOTV},
\cite{schatz-zambon:equivalence}.
It is known from \cite{zambon} that the above moduli problem is
obstructed in general, and the relevant obstruction lies in $H^2$
of the foliation de Rham cohomology of the relevant null-foliation $\CF_\omega$ of the
presymplectic manifold $(Y,\omega)$ \cite{oh-park:coisotropic} . 

\begin{rem}One may regard the above defined virtual dimension 
as the actual dimension of the moduli space of coisotropic submanifolds when 
the problem is obstructed. It would be an interesting problem to precisely describe 
this moduli space in a concretely given presymplectic manifold such as the one 
described in \cite[Section 12]{oh-park:coisotropic} or in the toric case \cite{weidong-ruan}.
\end{rem}

\part{Presymplectic stratification of a generic closed two-form}

In this part, we consider any closed two-form $\omega$ on a closed manifold $M$.
We prove a generic transversality result which reads that there exists a residual subset
of $\omega$'s which defines a Whitney stratification of presymplectic submanifolds.
We do this first by considering the canonical stratification of the set of
skew-symmetric bilinear forms on $\R^N$ and deriving the precise dimension formulae 
for the strata of  the bilinear forms in terms of  the dimension of rank (or equivalently the nullity) thereof,
and then applying to Sard-Smale theorem by considering a two form $\omega$ as 
a section of skew-symmetric bilinear forms over $M$,
\be\label{eq:quad-bundle}
\Lambda^2(M): = \Lambda^2(T^*M) \to M
\ee
whose fiber is the set $\Lambda^2(T_p^*M),\, p \in M$ consisting of 
skew-symmetric bilinear forms in $T_pM \cong \R^N$.

\section{Stratification of the set of skew-symmetric bilinear forms}

Let $N > 0$ be a positive integer and consider the set 
$$
\Lambda^2(V) =: S
$$
consisting of $N \times N$ skew-symmetric bilinear forms $Q$ on 
the $N$-dimensional vector space. We define its \emph{kernel} by
\be\label{kernel}
\ker Q: = \{ v \in V \mid Q(v, w) = 0\, \,  \forall \, w \in V\}
\ee
and call its dimension the \emph{nullity} of $Q$.
The set admits a decomposition
\be\label{eq:decomposition}
\bigcup_{k =0, \cdots N} S_m; \quad S_m: = \Lambda^2_m(V)
\ee
where $ \Lambda^2_m(V)$ is the subset of $\Lambda^2(V)$ defined by
$$
\Lambda^2_m(V): = \{ Q \in \Lambda^2(V) \mid \text{\rm nullity}(Q) = m\}.
$$
We denote by $S_Q$ the stratum $S_m$ containing $Q$. It follows from the upper semi-continuity of
the nullity that the sub-union
$$
\bigcup_{k \geq \ell} S_m \subset S
$$
is a closed subset of $S$ for each integer $\ell \geq 0$.  We denote by $\CS$
the set of strata of $\Lambda^2(V)$.
We note 
$$
\bigcup_{k \geq N+1} S_m = \emptyset.
$$

The following is an immediate consequence of the properties of general Whitney stratification.

\begin{lem}\label{lem:Whitney} Let $V$ be a vector space and 
consider the set $\Lambda^2(V)$. Let $<$ be the partial order between the strata of 
the decomposition \eqref{eq:decomposition}.  Then it defines a Whitney stratification.
\end{lem}
\begin{proof} Note that each stratum is an subanalytic subset, the inclusion maps
$$
\bigcup_{k \in \CS_1} S_m \hookrightarrow \bigcup_{k \in \CS_2} S_m
$$
are subanalytic maps for any $\CS_1 \subset \CS_2 \subset \CS$.
The lemma immediately follows from these observations.
(See \cite[Section 4]{mather:mappings}, \cite[Section 1.1 \& 1.2]{goresky-macpherson} for details.)
\end{proof}

Now we compute the dimensions of the strata explicitly.
We first recall that when the bilinear form $Q$ is nondegenerate $\dim V = 2n$ for some $n$ and that it is well-known
\be\label{eq:dim-Qdsk'}
\dim \Lambda(V) = \frac{2n(2n-1)}{2} = n(2n-1)
\ee
We also fix an auxiliary Euclidean inner product on $V$ and denote by $W^{\perp_2}$ the orthogonal
complement of $W$ in $V$ with respect to the inner product. On the other hand, we denote by $W^\perp \subset V^*$
the annihilator of $W$.

Let $m$ be the nullity of $Q$. Then we have the decomposition
\be\label{eq:V-decompose}
V = N_Q \oplus N_Q^{\perp_2}, \quad Q = 0 \oplus Q'
\ee
where  $Q'$ is a nondegenerate skew symmetric bilinear  form on $N^\perp$. We denote by $Q^\pi$
the projection of $Q$ to $N_Q$.

Therefore we have the following
lemma.

\begin{lem} Suppose $\dim V = N$ and let $Q \in \Lambda^2_m(V)$ with nullity $m$. Then we have
$$
\dim \Lambda^2_m(V) = \dim \operatorname{Gr}(\R^N,\R^m) + \frac{(N-m)(N-m-1)}{2}.
$$
\end{lem}
\begin{proof} The first factor arises from the choice of the null space $N_Q$ in $V$ and the second
factor follows from \eqref{eq:dim-Qdsk'} applied to $N_Q^{\perp_2}$.
\end{proof}
Recall 
$$
\dim \operatorname{Gr}(\R^N,\R^m) = \frac{N(N-1)}{2} - \left(\frac{m(m-1)}{2} +
\frac{(N-m)(N-m-1)}{2}\right) = m(N-m).
$$
We summarize the above discussion into the following.

\begin{prop} Let $m$ be the nullity and write $N = m + 2\ell$ for $\ell$ with $0 \leq \ell \leq \frac{[N]}{2}$.
Then we have
\be\label{eq:dimLambda2m}
\dim \Lambda^2_m(V) = 
m(N-m) + \frac{(N-m)(N-m-1)}{2} =  \frac{ (N-m)(N+m-1)}{2} 
\ee  
and  
 \be
 \text{\rm codim}\,  \Lambda^2_m(V)  =  \frac{m(m-1)}{2}
 \label{eq:codim-m} 
 \ee
for all $m$.
 \end{prop}
 One can check when $m = 0$ which corresponds to the open stratum.
On the other hand, it is interesting to see the codimension of the next stratum.

\begin{exm}\label{exm:m-constraint} We examine a few cases of $(m,\ell)$ with $m + 2\ell \leq N$
so that $0 \leq \ell \leq [N/2]$.
\begin{enumerate}
\item
When $m=0$, $N$ must be even, i.e., $N = 2k$ for some integer $k\geq 0$. Then $\ell = k$. 
\item When $m = 1$, $N$ must be odd and $N = 2k+1$ for some integer $k \geq 0$. Then $\ell = k$
and $\codim \Lambda^{2k+1}_1(V) = 0$.
\item If $m = 2$, then we have
$$
\codim  \Lambda^2_m(V) = 1.
$$
\item More generally  $\Lambda^2_m(V)$ has its codimension  given by  $\frac{m(m-1)}{2}$. \end{enumerate}
 \end{exm}

\begin{cor} The stratum next to the open strata has codimension $1$ and the rest of strata 
has  codimension greater than 2. 
\end{cor}
 
\section{Stratawise transversality of generic closed two-forms}

\subsection{Off-shell framework for the transversality study}

We regard a closed two-form $\omega$ as a section of this bundle
that satisfies the first order differential equation 
\be\label{eq:domega=0}
d\omega = 0
\ee
on $M$ by considering $d$ as a first-order differential operator 
$$
d: \Omega^2(M)\to \Omega^3(M).
$$
We consider the bundle of skew-symmetric bilinear forms
$$
\Lambda^2(M) \to M
$$
the fiber of which at $p \in M$ is given by $\Lambda^2(T_pM)$, and hence with rank $\frac{N(N-1)}{2}$.
We also consider $\Lambda^3(M)$.
Then we consider the fiber bundle
\be\label{eq:Lambda2mM}
\Lambda^2_m(M): = \bigcup_{x \in M} \{x\} \times \Lambda^2_m(T_xM):
\ee
The bundle $\Lambda^2_m(M) \to M$ is a fiber bundle whose fiber is given
by the skew-symmetric bilinear form of nullity $m$ $\Lambda^2_m(\R^N)$.

We have the decomposition
\be\label{eq:Lambda2M}
\Lambda^2(M) = \bigcup_{0 \leq m \leq [N/2] - 2k} \Lambda^2_m(M)
\ee
induced by \eqref{eq:decomposition},  where $2k$ is the maximal rank of the bilinear  
forms $\omega_x$, i.e.,
$$
2k = \max_{x \in M} \{ \rank \omega_x\}.
$$
Obviously, for every two-form $\omega$, we have
\be\label{eq:omegax-union}
\Lambda^2(T_x M) = \bigsqcup_{m = 0}^N \Lambda^2_m(T_x M).
\ee
\begin{cor}\label{cor:existence} For every nonzero two-form $\omega$,
there is some $m < N$, such that 
$$
\omega_x \in \Lambda_m^2(T_xM).
$$
\end{cor}

We remind the readers that $\Omega^2(M)$ is an
infinite dimensional vector space.
Then we consider the differential operator
$$
d: \Omega^2(M) \to \Omega^3(M)
$$
and whose kernel, i.e., the subset of closed differential two-forms,
which we denote by $Z^2(M)$ as usual.

\begin{thm}\label{thm:transversality} We  consider the evaluation map 
$$
\ev: M \times \Omega^2(M) \to\Lambda^2(M); \quad \ev(x,\omega): = \omega_x.
$$
Then for each nullity $m$, $\ev$
is transverse to the stratum $\Lambda^2_m(M)$ at all $(x,\omega)$ satisfying
$$
d\omega = 0, \quad \omega_x \in \Lambda^2_m(T_xM).
$$
\end{thm}
The proof of this theorem will be postponed till the end of the
the present section. 

An immediate corollary of this transversality theorem is the following
\begin{cor} For each $m$, the preimage
$$
\ev^{-1}\left(\Lambda^2_m(M) \right)  \cap (M \times Z^2(M)) \subset M \times \Omega^2(M)
$$
is a (infinite dimensional) smooth submanifold of $M \times \Omega^2(M)$.
\end{cor}

\subsection{Parametrization of closed two-forms}

We fix a Riemannian metric $g$ on $M$ and consider the Hodge decomposition
$$
\Omega^2(M) = H^2(M) \oplus d\Omega^1(M) \oplus \delta \Omega^3(M).
$$
We start with recalling the following consequence of the Hodge decomposition.
\begin{prop} Let $\omega$ be a given closed two-form. Then any other 
closed two form in the same de Rham cohomology class $[\omega]$
can be expressed as
$$
\omega' = \omega_h + d\beta, \quad \omega_h: =  \CH(\omega)
$$
for some $\beta \in \delta\Omega^1(M)$.
\end{prop}

To provide the the structure of Banach manifolds on the set of closed two-forms 
which we will need later in the next subsection, we provide some details of this decomposition now.

We have
$$
Z^2(M) = H^2(M) \oplus d\Omega^1(M)
$$
and a given closed two form $\omega$ can be decomposed into
$$
\omega = \CH(\omega) + d\beta
$$
for the $\CH(\omega)$ the harmonic projection of $\omega$.  WLOG,
we will fix the cohomology class $[\omega] \in H^2_{\text{\rm dR}}(M,\R)$ and then we can parameterize 
a $C^\infty$ neighborhood of any given closed two-form by the set
$$
d\Omega^1(M) = d (\Omega^1(M) / Z^1(M)).
$$
Again by the Hodge decomposition 
$$
\Omega^1(M) = H^1(M) \oplus d\Omega^0(M) \oplus \delta \Omega^2(M),
$$
we have
$$
\Omega^1(M)/Z^1(M) \cong \delta \Omega^2(M)
$$
and hence
$$
d\Omega^1(M) \cong d \delta \Omega^2(M) \cong d \delta d\Omega^1(M).
$$
Furthermore $Z^2(M)$ is decomposed into the union 
$$
Z^2(M) = \bigsqcup_{h \in H^2(M)} Z^2_h(M), \quad Z^2_h(M): =  \{\omega \in Z^2(M) \mid [\omega] = h\}
$$
of affine spaces
$$
Z^2_h(M) =  \omega_h + B^2(M)
$$
where $\omega_h$ is the unique harmonic 2-form in class $h \in H^2_{\text{\rm dR}}(M)$.
We denote by 
$$
\mathsf H_g^2 = \mathsf H^2(M;g) \cong \R^{b_2(M)}
$$
 the set of harmonic 2-forms of the given Riemannian metric $g$ of $M$.
By this proposition, the set $Z^2(M)$ is a smooth Frechet manifold modeled by 
\be\label{eq:tangent-space}
\mathsf H^2(M) \oplus B^2(M) \cong 
\R^{b_2} \oplus B^2(M) 
\ee
such that the tangent space $T_\omega Z^2(M)$ can be canonically 
identified with \eqref{eq:tangent-space}.

\subsection{Fredholm property of the parameterized linearization of $\Upsilon$}

To establish the transversality stated in  Theorem \ref{thm:transversality}, we
consider the following map
$$
\Upsilon: M \times \Omega^2(M) \to  \Lambda^2(M) \times B^3(M)
$$
defined by
\be\label{eq:Upsilon}
\Upsilon(x,\omega) : = (\omega_x, d\omega).
\ee
Here we recall the definition of the vector bundle $\Lambda^2(M) \to M$ where
$\Lambda^2(M)$ is the finite union \eqref{eq:Lambda2M} of (finite-dimensional) vector bundles
\eqref{eq:Lambda2mM}.

We  will need to prove transversality of the map against
$$
\Lambda^2_m(M) \times \{0\} \subset \Lambda^2(M) \times B^3(M).
$$
Here $B^3(M)$ is the set of exact 3-forms:
By definition, the map has its natural codomain
contained in $\Lambda^2(M) \times B^3(M)$, and we have the includsion
$$
\Upsilon^{-1}(\Lambda^2_m(M) \times \{0\}) 
\subset M \times Z^2(M).
$$
In fact, by definition, we have  
\be\label{eq:Promega}
\ev^{-1}\left(\Lambda^2_m(M)\right) \cap (M \times Z^2(M)) 
= \Upsilon^{-1}\left(\Lambda_m^2(M) \times \{0\}\right).
\ee

Let $(x,\omega) \in M \times \Omega^2(M)$ such that
$d\omega = 0$ and $\omega_x \in \Lambda^2_m(M)$ for some $m = 0, \, 2, \cdots$. We have
$$
T_{(x,\omega)}(M \times \Omega^2(M)) = T_xM \oplus \Omega^2(M)
$$
and 
$$
T_{(\omega_x,d\omega)} (\Lambda^2(M) \times B^3(M))
= \Lambda^2(T_x^*M) \oplus B^3(M).
$$
The following is straightforward to check.
\begin{lem}\label{lem:linearization}
The  linearization of $\Upsilon$ at $(x,\omega)$ is given by
\be\label{eq:dUpsilon}
d\Upsilon(x, \omega): T_xM \oplus \Omega^2(M) \to \Lambda^2(T_x^*M) \oplus B^3(M)
\ee
where we have
\be\label{eq:dUpsilon-formula}
d\Upsilon(x, \omega)(v, \alpha) = (\nabla_v \omega + \alpha_x, d\alpha).
\ee
for $\alpha \in \Omega^2(M)$.
\end{lem}

The following is the main transversality result of the present subsection.
\begin{prop}\label{prop:transversality}
The map $\Upsilon$ is transversal to 
$$
\Lambda^2_m(M) \times \{0\}
$$
in $\Lambda^2(M) \times B^3(M)$ for each nullity $m$. 
\end{prop}
\begin{proof} 
Let $(x,\omega) \in M \times \Omega^2(M)$ such that
$d\omega = 0$ and $\omega_x \in \Lambda^2_m(T_xM)$. We need to show
\be\label{eq:transversality}
\Image d_{(x,\omega)}\Upsilon + \left(T_{\omega_x} (\Lambda^2_m(M)) \times \{0\} \right)
= T_{\omega_x}( \Lambda^2(M)) \times T_\omega( B^3(M))
\ee
at all $(x,\omega) \in\Upsilon^{-1}(\Lambda^2_m(M)\times \{0\})  \subset M \times Z^2(M)$.

For this purpose, we  apply
the Fredholm alternative and the Hahn-Banach theorem (after suitably completing
the codomain of the linearization map $d_{(x,\omega)}\Upsilon$, after taking a suitable
completion of the domain and the codomain of the linear map $d\Upsilon(x,\omega)$ of \eqref{eq:dUpsilon}).  
Noting that 
$d\Upsilon(x, \omega)$ is a first-order differential operator as given in the formula \eqref{eq:dUpsilon-formula}, 
we will take the $W^{1,p}$ and $L^p$ 
completions of the domain and of the codomain respectively for a suitable choice of $p > 1$.
First of all, for the evaluation map $\omega \mapsto \omega_x$ to make sense and  to be continuous,
we need to choose
\be\label{eq:p>N}
p> N = \dim M
\ee
by the Sobolev embedding $W^{1,p} \hookrightarrow C^0$. (See \cite[Theorem 7.10]{gilbarg-trudinger}, 
for example.)

Suppose 
$$
(Q, \eta) \in T_{\omega_x}( \Lambda^2(M)) \times T_\omega(B^3_{L^p}(M))
$$
is contained in
\bea
&{} &\left(\Image d_{(x,\omega)}\Upsilon +  (T_{\omega_x} (\Lambda^2_m(M)) \times B^3(M))\right)^{\perp_2}
\label{eq:etainL2} \\
&= & (\Image d_{(x,\omega)}\Upsilon)^{\perp_2}  \cap  (T_{\omega_x} (\Lambda^2_m(M)) \times B^3(M))^{\perp_2} 
\nonumber \\
& \subset & T_{\omega_x}\Lambda^2(M) \times B^3_{L^q}(M) \nonumber
\eea
with $\frac1q + \frac1p = 1$.
 
Recalling $\eta  \in T_{d\omega} B^3_{L^p} (M)= T_{0} B^3(M) \cong B^3(M)$, it also
weakly satisfies $\eta = d \zeta$ for some $\zeta \in \Omega^2_{W^{1,p}}(M)$. 
This makes \eqref{eq:L2-cokernel} is equivalent to
\be\label{eq:L2-cokernel}
\langle \nabla_v \omega + \alpha_x, Q_x \rangle  
+  \int_M \langle d\alpha, \eta \rangle \, d\vol _g = 0
 \ee
 for all $\eta = d \zeta$ with $\zeta \in \Omega^2_{W^{1,p}}(M)$.
In addition, $(Q,\eta) \in  (T_{\omega_x} (\Lambda^2_m(M)) \times B^3_{L^p}(M))^{\perp_2}$ is 
equivalent to
$$
\langle \kappa_x, Q_x \rangle = 0 \quad \forall \kappa \in T_{\omega_x} (\Lambda^2_m(M)).
$$
The latter implies 
\be\label{eq:Q-parallel}
Q_x \in (T_{\omega_x}\Lambda^2(M))^{\perp}
\ee
i.e., $Q_x^\parallel = 0$ in the decomposition of $Q_x$ into 
$$
Q_x =  Q_x^\parallel + Q_x^\perp
$$
according to the orthogonal decomposition
$$
T_{\omega_x} (\Lambda^2(M)) = T_{\omega_x}(\Lambda_m^2(M)) \oplus N_{\omega_x}(\Lambda_m^2(M)).
$$
Substituting this back into \eqref{eq:L2-cokernel} and using the special connection given in Theorem \ref{thm:omega-special} gives rise to
$$
\langle \alpha_x^\perp, Q_x^\perp \rangle + \int_M \langle d\alpha,d\zeta \rangle \, d\vol _g = 0
$$
for all smooth $\alpha \in \Omega^2(M)$.  
 Therefore  $\eta \in L^q$ is a distributional solution of the equation
 \be\label{eq:delta-eta=0}
 \delta \eta + Q^\perp \delta_x = 0.
 \ee
 
Recalling that we take the $L^p$-completion of the image of the operator \eqref{eq:dUpsilon}, 
we can express $\eta = d\zeta$
for a $W^{1,q}$-form $\zeta$ by the Sobolev space version of the Hodge decomposition theorem. 
(See \cite[Lemma 2 \& 3, p.124]{deRham}, for example.) Without loss of 
generality,  again applying the Hodge decomposition,
we may assume $\delta \zeta = 0$ and $\zeta \in W^{1,q}$ as $d\zeta' = d\zeta$
when $\zeta' = \zeta + \delta$ for any closed 2-form $\delta$.
Then \eqref{eq:delta-eta=0} is reduced to
\be\label{eq:zeta}
\Delta \zeta + Q_x^\perp \delta_x = 0.
\ee
 If $Q_x^\perp = 0$, then $\zeta$ satisfies $\Delta \zeta = 0$ weakly and so a smooth
 harmonic form on $M$. In particular $\eta = d\zeta = 0$.
 
 Therefore we now assume $Q^\perp_x \neq 0$ and so has rank at least 2 by skew-symmetry.
 In this case, $\zeta$ is a fundamental solution of $\Delta$ multiplied by $C Q^\perp$ 
for some constant $C$, i.e., we have
$$
\zeta = C\,  G \, Q^\perp
$$
with the Green form $G$ in a coordinate neighborhood $U$ 
with coordinates centered at $x$ where $Q^\perp \equiv Q^\perp_x$
is the constant bilinear form on $U$.

 We postpone the proof of the following lemma till Appendix \ref{sec:eta=0}.
\begin{lem}\label{lem:eta=0} The one-form $\zeta$ extends smoothly across $x$ and so a smooth
harmonic one-form. In particular, we have $\eta = d\zeta = 0$.
\end{lem}

This lemma in particular implies 
\be\label{eq:eta=0}
\eta = d\zeta = 0
\ee
 in $ B^3_{L^p}(M)$. 
Substituting back this into the equation \eqref{eq:L2-cokernel}, we also prove
$$
\langle \alpha_x^\perp, Q_x^\perp \rangle = 0
$$
for all smooth $\alpha$, which implies 
\be\label{eq:Q-perp}
Q_x^\perp  = 0.
\ee
Combining \eqref{eq:Q-parallel}, \eqref{eq:eta=0} and \eqref{eq:Q-perp}, we have proved
$(Q,\eta) = 0$. This  proves that the $L^2$-cokernel \eqref{eq:etainL2} is trivial,
which finishes the proof of 
\eqref{eq:transversality} and hence the proposition.
 \end{proof}
 
\section{Application of Sard-Smale theorem and existence}
  
Let $m$ be fixed and consider the projection map
$$
\Pi_m: \Upsilon^{-1}(\Lambda^2_m(M) \times \{0\}) \to Z^2(M):
$$
\subsection{Banach manifold set-up}

In this subsection, we consider the $C^k$ forms $\omega$ for $k \geq 2$,
and extend the map $\Upsilon$ on $M \times C^k(\Lambda^*(M))$ and take the
completion of $d$ thereto 
$$
d:\Omega_{C^k}^*(M) \to \Omega_{C^{k-1}}^{*+1}(M).
$$
(See \cite{deRham} for the detailed study of this extended operator and the associated
elliptic complex.) We denote by $\Upsilon_k$ this $C^k$-extension of $\Upsilon$.

We denote by 
\bea\label{eq:ZakBak}
Z^a_{C^k}(M) &: = & \ker \left(d:\Omega^a_{C^k}(M) \to 
\Omega^{a+1}_{C^{k-1}}(M)\right), \nonumber \\
 B^a_{C^k}(M) &: = & \Image \left(d:\Omega_{C^{k+1}}^{a-1}(M) \to \Omega_{C^k}^a(M) \right).
\eea
This being said, we consider  the restriction
\be\label{eq:Pimk}
\Pi_m^k:  \Upsilon_k^{-1}(\Lambda^2_m(M) \times \{0\}) \subset  M \times Z^2_{C^k}(M) \to Z^2_{C^k}(M)
\ee
of the $C^k$-extension of the obvious projection map
$$
M \times \Omega_{C^k}^2(M) \to \Omega_{C^k}^2(M).
$$

The following Fredholm property is a key element of the proof.
\begin{prop} The map $\Pi_m$ in \eqref{eq:Pimk} is a Fredholm map.
\end{prop}
\begin{proof} We consider the linearization map
$$
d\Pi_m(x,\omega): T_{(x,\omega)} \Upsilon^{-1}\left(\Lambda^2_m(M) \times \{0\}\right)
 \to Z^2_{C^k}(M)
$$
which is a linear map from 
\bea\label{eq:tangent-space-Upsilon}
&{}& T_{(x,\omega)} \Upsilon^{-1}(\Lambda^2_m(M) \times \{0\}) \nonumber\\
& = & \{(v,\alpha) \in  T_xM \oplus Z^2_{C^k}(M) \mid
\nabla_{v_i} \omega + \alpha_i|_x \in T_x \Lambda_m(M), \quad d \alpha_i = 0 \}
\eea
by the preimage value theorem applied to $\Upsilon$, where $\Omega^2(M)$
is equipped with $C^k$-topology. 

Since the fiber of the projection map \eqref{eq:tangent-space-Upsilon} is of finite dimension,
it is enough to prove that the image of $d\Pi_m$ is closed.

Let $\alpha_i \in Z^2_{C^k}(M)$ be a convergent sequence in $Z^2_{C^k}(M)$ in $C^k$-topology
and let $\alpha_\infty$ be its $C^k$-limit. We will show that $\alpha_\infty \in \Image d\Pi_m(x,\omega)$, i.e.,
we have to find $v$ such that
$$
\nabla_{v} \omega + \alpha_\infty|_x \in T_x\Lambda^2_m(M), \quad d \alpha_\infty = 0.
$$
By \eqref{eq:tangent-space-Upsilon} and Theorem \ref{thm:omega-special} in Appendix, this implies
$$
\alpha_i|_x \in T_x\Lambda^2_m(M), \quad d \alpha_i = 0
$$
by choosing a special connection as therein. Therefore we can set $v_i = 0$. Therefore we can put $v = 0$.
Then we have $\alpha_\infty = d\Pi_m^k(0,\alpha_\infty) \in \image d\Pi_m^k$ which
 proves the closedness of the image of $d\Pi_m^k$.
\end{proof}

Now by applying Sard-Smale theorem, we have proved the following

\begin{cor}\label{cor:generic-closed2forms} For the set of regular values $\omega$ of $\Pi_m^k$, the preimage
$$
(\Pi_m^k)^{-1}(\omega) = \ev_\omega^{-1}(\Lambda^2_m(M)),
$$
is a \emph{nonempty} smooth finite dimensional manifold for which the map 
$$
\ev : M \to \Lambda^2(M)
$$
is transversal to $\Lambda^2_m(M)$ for all $m = 0, \, 2, \cdots$.
\end{cor}
We denote by 
$$
Z^2_{C^k;\text{\rm reg}}(M) : = \bigcap_{m=0}^N \text{\rm Reg}(\Pi_m^k),
$$
 the intersection of the set $\text{\rm Reg}(\Pi_m^k)$ consisting of regular values of 
$\Pi^k_m$ over $m = 0, \cdots, N$.

To derive the same conclusion for the $C^\infty$ case we have only to first take a countable intersection
$$
\bigcap_{k = 1}^\infty Z^2_{C^k;\text{\rm reg}}(M) =: Z^2_{\text{\rm reg}}(M).
$$
\subsection{Existence of nice closed two-forms}

By the general property of Whitney stratifications (see \cite{mather:mappings}, 
\cite[Section 1.2]{goresky-macpherson}), we have proved that the collections
$$
\ev_\omega^{-1}\left(\Lambda^2_m(M)\right), \quad m = 0, 2, \cdots, 2[N/2],
$$
define a Whitney stratification on $M$, \emph{provided the non-emptiness holds in Corollary \ref{cor:generic-closed2forms}}.

Now the study of this non-emptiness result is in order. By definition, we have
$$
(\Pi_m^k)^{-1}(\omega) = \{(x,\omega) \in M \times Z^2_{C^k}(M) \mid d\omega = 0, \, \omega_x \in \Lambda^2_m(T_xM),\,
\omega \, \text{\rm is\, } C^k\}.
$$
The following ampleness of the set of closed two-forms is crucial for our proof of the existence of
nice closed two-forms.

\begin{lem} Let $N = \dim M$.
Let $k \geq 1$ and $x_0 \in M$ be any point. Consider the subset
$$
\Lambda^2_{m;\text{\rm cl}}(T_x^*M) = \{Q \in \Lambda^2(T_x^*M) \mid \omega \in  Z^2_{C^k}(M), \, Q = \omega_x \in \Lambda^2_m(T_xM),\,
\omega \, \text{\rm is\, } C^k\}.
$$
Then we have
$$
\Lambda^2(T_x^*M) =  \span_{m=0}^{N-1} \left\{\Lambda^2_{m;\text{\rm cl}}(T_xM) \right\}.
$$
\end{lem}
\begin{proof} Fix a coordinates $(x_1, \cdots, x_N)$ on a neighborhood $U$ centered at $x_0$.
We have
$$
Q = \sum_{i>j} Q_{ij} dx_i \wedge dx_j|_{x_0}
$$
for some  coefficients $Q_{ij}$ that are skew-symmetric.

Let $\omega$ be a closed two-form and express it as
$$
\omega = \sum_{i>j} \omega_{ij}(x) dx_i \wedge dx_j
$$
on $U$. Then we have
$$
\frac{\del \omega_{ij}}{\del x_k} + \frac{\del \omega_{jk}}{\del x_i} + \frac{\del \omega_{ki}}{\del x_j} = 0.
$$
We want to find such $\omega$ solving the equation
\be\label{eq:omegax=Q}
\omega_{ij}|_{x_0} = Q_{ij}.
\ee
For this purpose, every closed two-form is locally exact and so can be expressed as
$\omega = d\alpha$ on $U$ for some one-form $\alpha = \sum_{i=1}^N \alpha_i \, dx_i$ so that
$$
\omega_{ij} = \frac{\del \alpha_i}{\del x_j} - \frac{\del \alpha_j}{\del x_i}
$$
for $\omega_{ij}$ satisfying $\omega_{ji} = -\omega_{ij}$.
Therefore the equation \eqref{eq:omegax=Q} is reduced to solving the equation
$$
\frac{\del \alpha_i}{\del x_j}(x_0) - \frac{\del \alpha_j}{\del x_i}(x_0) = Q_{ij}
$$
for $\alpha$. This can be solved by taking
$$
\alpha_i =\frac14 (Q_{ik}x_k - Q_{\ell i}x_\ell)
$$
(with Einstein's summation convention exercised). We compute
$$
\frac{\del \alpha_i}{\del x_j}(x_0) = \frac14 (Q_{ij} - Q_{ji}) = \frac12 Q_{ij}.
$$
Similarly we compute
$$
\frac{\del \alpha_j}{\del x_i}(x_0) = \frac12 Q_{ji} = -\frac12 Q_{ij}.
$$
This shows
$$
\frac{\del \alpha_i}{\del x_j}(x_0) - \frac{\del \alpha_j}{\del x_i}(x_0) = \frac12 Q_{ij} - (-\frac12 Q_{ij}) = Q_{ij}.
$$
This proves that the one-form $\alpha = \sum_{i=1}^N \alpha_i\, dx_i$ indeed solves the equation.
In addition, we note that if $Q$ has maximal nullity i.e., $N = \dim M$, we have $\Lambda^N(TM) = \{0\}$.
Combining the two, we have  finished the proof.
\end{proof}

We now obtain the following non-emptiness result.
\begin{prop} We have
$$
\span \left\{\ev((\Pi_m^k)^{-1}(\omega))|_x \right\}_{m=0}^{N-1} = \Lambda^2(T_x^*M).
$$
In particular $(\Pi_m^k)^{-1}(\omega) \setminus \{0\} \neq \emptyset$ 
for some $m$ with $0 \leq m \leq N-1$.
\end{prop}
\begin{proof} By definition, we have
$$
\ev((\Pi_m^k)^{-1}(\omega))|_x = \{\omega_x\} \cap \Lambda^2_{m;\text{\rm cl}}(T_x^*M).
$$
By the above lemma, we know
$$
\span \left\{\Lambda^2_{m;\text{\rm cl}}(T_x^*M)\right\}_{m=0}^{N-1} = \Lambda^2(T_x^*M).
$$
Therefore $\omega_x \in \Lambda^2_{m;\text{\rm cl}}(T_x^*M)$ for some $m$, which finishes the proof.
\end{proof}

Introduction of the following notion is useful for further discussion.
 
 \begin{defn}[Nice closed two-forms]\label{defn:nice-forms}
 We denote by $Z^2_{\text{\rm reg}}(M)$ the set of closed 2-forms
that satisfy the hypothesis of Corollary \ref{cor:generic-closed2forms}. We call any 
such closed two-form $\omega$ a \emph{nice} closed two-form, \emph{provided non-emptiness holds.}
\end{defn} 
 By definition, the set of nice closed two-forms is a residual subset of $Z^2(M)$.
 
\begin{thm} \label{thm:adaptedness}
Let $\omega$ be a nice closed two-form on $M$. Then $M$ carries a Whitney stratification
$\CS$ whose strata are given by
$$
M = \bigcup_{m=0}^N Y_m; \quad Y_m = \ev_\omega^{-1}\left(\Lambda^2_m(M)\right)
$$
such that
\begin{enumerate}
\item Each $Y_m$ is locally closed and satisfies $\overline Y_m \supset Y_{m+2}$.
\item  $\omega$ has constant nullity $m$ on each $Y_m$, and so is presymplectic 
of rank $N-m$ thereon.
\end{enumerate}
\end{thm}
\begin{proof} We start with Lemma \ref{lem:Whitney} which in turn induces Whitney stratification
$$
\Lambda^2(M) = \bigcup_{m = 0}^{N} \Lambda^2_m(M).
$$
By definition of nice closed two-forms, the map
$$
\ev_\omega: M \to \Lambda^2(M)
$$
defined by $\ev_\omega(x) = \omega_x$ is transverse to $\Lambda^2_m(M)$
for all $m$, i.e., is stratawise transverse. By the stratawise preimage value theorem, 
the decomposition
$$
M = \bigcup_{m=0}^N (\ev_\omega)^{-1}(\Lambda^2_m(M))
$$
defines a Whitney stratification of $M$.
This finishes the proof.
\end{proof}

We recall that if $N < \frac{m(m-1)}{2}$, then $\dim M < \codim \Lambda^2_m(M)$
which implies that the image
$$
(\Image \ev_\omega) \cap \Lambda^2_m = \emptyset.
$$
\begin{rem}
 In particular, we always have
\be\label{eq:nullity<}
0 \leq m < \min \left\{\frac12 + \sqrt{8N + \frac14}, N\right \} = \frac12 + \sqrt{8N + \frac14}
\ee
for a nice closed two-form, where the equality always holds.
In other words, there is some hidden
constraint, \emph{except when $N = 2$}, for the nullity of a nice two-form $\omega$ than usual in that the nullity 
$\omega_x$ cannot lie in the region
$$
\left[ \frac12 + \sqrt{8N + \frac14}, N \right].
$$
\end{rem}

\part{Compatibility of $L_\infty$ structures: stratified $L_\infty$ spaces}

In this section, we interpret each of the $L_\infty$ structures 
associated to the strata of presymplectic submanifolds of $M$ 
as a \emph{virtual} symplectic neighborhood, and glue them to
define a \emph{global} structure of \emph{stratified $L_\infty$ space} or simply
an $L_\infty$ space.

\section{Review of the strong homotopy Lie algebroids in \cite{oh-park:coisotropic}}

In this section, we recall the construction of an $L_\infty$ structures associated to 
a presymplectic manifold from \cite{oh-park:coisotropic} for the original construction and others
such as \cite{cattaneo-schatz}, \cite{LOTV} 
for different and more systematic constructions in terms of the $V$-data 
of Voronov \cite{voronov1}.) 
 Here we follow the approach explained by Cattaneo-Sch\"atz \cite{cattaneo-schatz},  \cite{LOTV}
 which in turn follows that of  \cite{cattaneo-felder}. 

Let $\{ , \}$ be the Poisson bracket on $V \subset T^*\CF_\omega$ associated to the symplectic form
 $\omega_{T^*\CF_\omega}$ and denote by $P \in \Gamma(V;\Lambda^{2}(TE))$ for $E: = T^*\CF_\omega$,
the associated bivector field on $V$ which satisfies $[P,P]=0$ where $[\, ,\, ]$ is the Schouten-Nijenhuis bracket.
In this case, we have $E = T^*\CF_\omega$ and so
\begin{equation}\nonumber
N\sigma \cong T^*\CF_\omega,
\end{equation}
where we use the canonical identification $T_{(x,0)}T^*\CF_\omega = T_x\sigma \oplus T_x^*\CF_\omega$ 
for any section $\sigma \in \Gamma(E)$ such that $\Image \sigma \subset V$.
The following is proved in \cite{oh-park:coisotropic}. (See \cite[Corollary 4.6]{cattaneo-schatz} for 
a proof in the spirit of the present paper which is different from that of \cite{oh-park:coisotropic}.)

\begin{lem}[Theorem 9.4 \cite{oh-park:coisotropic}]\label{lem:strlinf}
$\{\mathfrak l^\omega_k\}$ is strict, i.e., $\mathfrak l^\omega_0 =0.$ Moreover, $l_1^{\CF_\omega}$ coincides with the foliation de Rham differential $d_{\CF_\omega}.$
\end{lem}
From now on, we write $\{l_k\}_{k \geq 1}$ for this (strict) $L_{\infty}[1]$-algebra, omitting $\CF_\omega$ when no confusion can occur.

For the later purpose of gluing a family of $L_\infty$ structures defined on the system of 
tubular neighborhoods of the strata $\{Y_j\}$, we need to provide more explicit
geometric description of the construction given in \cite{oh-park:coisotropic}.

Let $(Y, \omega_Y)$ be a presymplectic manifold and consider its null distribution 
$$
E : = \ker  \omega_Y \subset TY.
$$ 
The closedness of $\omega_U$ implies that $E$ is involutive and hence integrable by the Frobenius theorem. 
We denote by $\CF_\omega$ the associated null foliation of $U^n$ with $E = T \CF_\omega$.
We have an exact sequence
$$
0 \to E \to TY \to TY/E \to 0
$$
and the quotient $TY/E \to Y$ carries a canonical symplectic vector bundle, called the \emph{symplectic
normal bundle} of the foliation $\CF_\omega$.

For further discussions, introduction of the following term is useful.
\begin{defn}[Polarization] We call
 an Ehresmann connection, i.e., a splitting of the above exact sequence
\be\label{eq:splitting}
\Gamma: \quad TY = G \oplus T\CF_\omega
\ee
or the associated $G$ a \emph{polarization},  and denote by $\Pi_G: TY \to TY$ the associated idempotent. 
\end{defn}

We denote the set of 
such projections by
$$
\CA_E(TY) \subset \Gamma(Hom(TY,TY)) = \Omega^1_1(Y).
$$
The following is easy to see by using the isomorphism
$\pi_{G_0}: G_0 \to N\FF$ where $\pi_{G_0}$ is the restriction
to $G_0$ of the natural projection $\pi: TY \to N\FF$.
We omit its proof.

\begin{lem}\label{lem:splittings} Let $G_0$ be the slice in a given splitting $TY = G \oplus E_0$.
The space of splittings \eqref{eq:splitting} is the affine space
modeled by
$$
\Gamma(Hom(N\FF,T\FF)) \cong \Gamma(N^*\FF\otimes T\FF):
$$
for any reference choice $\Pi_0$ and other $\Pi$, we have
$$
G = \{y+ B_{\Pi_0\Pi}(y) \in TY \mid y \in G_0, B_{\Pi_0\Pi} \in
\Gamma(Hom(G_0,E)) \}
$$
In particular, it is contractible.
\end{lem}
Here $TY/E$ may be regarded as the tangent bundle of the leaf space $Y /\sim$ of the foliation. 
In general $Y/\sim$ is not a Hausdorff topological space.

\begin{defn}[Polarized presymplectic manifolds]\label{defn:polarized} Let $(Y,\omega)$ be a presymplectic manifold and
$G \subset TY$ be a choice of complementary subbundle of $\CN \subset TY$ in \eqref{eq:splitting}.
We call a triple $(Y,\omega, G)$ a \emph{polarized presymplectic manifold}.
\end{defn}
With this definition, Gotay's symplectic neighborhood in \cite{gotay} can be stated as
\begin{thm}[Gotay] Let $(Y,\omega)$ be a presymplectic manifold and $G \subset TY$ 
be a polarization $\Gamma$ thereof. Then the following holds:
\begin{enumerate}
\item There exists a sufficiently small neighborhood $V$ of the zero section
of the foliation cotangent bundle $T^*\CF_\omega$ on which
\be\label{eq:Gotay-form}
\omega_{T^*\CF_\omega;G} := \pi^*\omega - d \theta_G
\ee
is nondegenerate where $\theta_G$ is the fiberwise Liouville one-form of $T^*\CF_\omega$ The form $\omega_{T^*\CF_\omega}$.
\item For any coisotropic submanifold $Y \subset (M,\omega)$ for symplectic manifold $(M,\omega)$ with
polarization $G \subset TY$ for $(Y,\omega|_Y)$,
there exists a neighborhood $U \supset C$ in $M$, $V \subset o_{T^*\CF_\omega}$ in $T^*\CF_\omega$ and a symplectic
diffeomorphism $\Phi:U \to V$ such that 
$$
\omega = \Phi^*\omega_{T^*\CF_\omega;G}.
$$
\end{enumerate}
\end{thm}

\begin{rem}
In the point of view of $L_\infty$ Kuranishi structures 
as indicated in the introduction of the present paper,
one may regard Gotay's symplectic neighborhood of a presymplectic manifold $(Y,\omega)$
a \emph{virtual symplectic neighborhood} of $Y$ as a virtual neighborhood in the Gromov-Witten theory.
We highlight the fact that this virtual neighborhood is \emph{intrinsic} (modulo gauge equivalence)
which depends only on the structure of null foliation of $(Y,\omega)$. We will call the corresponding intrinsic
neighborhood a \emph{primary virtual neighborhood} in the notion of $L_\infty$ Kuranishi structure.
(See the definition in Appendix B.)
\end{rem}

\section{Compatibility of system of null foliations}
\label{sec:compatible-system-null}

Assume that  $(M,\omega)$ is equipped with a smooth manifold with a \emph{nice} closed 2-form $\omega$
in the sense of Definition \ref{defn:nice-forms}. By Corollary \ref{cor:generic-closed2forms}
and Theorem \ref{thm:adaptedness}, we have a Whitney stratification $\mathcal{S}$
$$
M = \bigcup\limits_{j = N - 2[N/2]} Y_j,
$$
where $Y_j := ev^{-1}_{\omega}\big(\Lambda^2_j(M) \big)$ so that the conditions (1) and (2) hold. 
We can order the set of strata $\{Y_j\}$ by
$$
Y_j < Y_{j'} \text{ if and only if } {M}_{j} \subset \overline{M}_{j'}.
$$
One can check that this relation is transitive from the frontier axiom of stratifications.

\subsection{Mather's compatible system of neighborhoods}

Recall the definition of a compatible system of tubular neighborhoods of $M$ that Mather employed in his study of
the structure of singularities of smooth maps in \cite{mather:mappings}.

Let $\iota: S \hookrightarrow M$ be a submanifold and $\pi : F\rightarrow S$ a vector bundle over
$S$ in $M$ equipped with an inner product smoothly on each fiber. This is a geometric realization of
the normal bundle $N(S;M) \to S$ in $M$.
For a positive smooth function $\epsilon$ on $S$ which we call a \emph{scale function}, 
we denote by $D_\epsilon(F)$ the $\epsilon$-disc bundle of $F$, that is, the set of all 
$v$ satisfying $\|v\| \leq \epsilon\big(\pi(v)\big).$ Then a tubular neighborhood of $U$ of $S$ 
in $M$ is defined by the map $\phi : D_\epsilon(F) \rightarrow M$,
 a diffeomorphism of $D_\epsilon(F)$ onto an open subset of $M$ with the property 
$\iota = \phi \circ \zeta$ and 
 \be\label{eq:tubular-nhbd}
\xymatrix{ D_\epsilon \ar[d]^{\pi} \ar[r]^\phi & M \ar@{=}[d]\\
  S \ar @{^{(}->}[r]_\iota \ar@<1ex> [u]^{\zeta} & M
}
\ee
where $\zeta : S \rightarrow D_\epsilon(F) \subset F$ is the zero section embedding.
Such a map always can be constructed by the \emph{rescaled} normal exponential map by the 
scale function $\epsilon$.

We further denote $U= U_\epsilon := \phi(D_{\epsilon}(F)).$ We call the retraction map
\begin{equation}\nonumber
\pi_U := \pi \circ \phi^{-1} : U \rightarrow S, 
\end{equation}
the \textit{projection associated to} $T$ and the positive function
\begin{equation}\nonumber
\rho_U := \rho \circ \phi^{-1} : U \rightarrow \mathbb{R}_{\geq 0}
\end{equation}
the \textit{tubular function associated to} $U_\epsilon$. Here $\rho$ is the quadratic function
\be\label{eq:rho}
\rho(v) := \|v\|^2.
\ee 
for $v \in D_\epsilon(F)$
Observe that the pair of maps
\begin{equation}\nonumber
(\pi_U, \rho_U) :U  \rightarrow S \times \mathbb{R}_{\geq 0}
\end{equation}
is a submersion from the definitions.

This enables us to choose the following system of compatible tubular neighborhoods
of Mather  \cite{mather:mappings}

\begin{prop}[Mather \cite{mather:mappings}] Let $\{Y_j\}$ be the above dimension stratification of $M$.
Then there is a system of tubular neighborhood of the strata 
$$
(\pi_{U_j}, \rho_{U_j}): U_j \to Y_j \times \R_{\geq 0}
$$
satisfying the following compatibilities:
Let $\{U_j \supset Y_j\}_j$ be a family of neighborhoods of $Y_j$  induced by
a system of tubular neighborhoods
$$
\left(\pi_{U_j} : U_j  \rightarrow Y_j, \rho_{U_j} : U_j \rightarrow \mathbb{R}_{\geq 0}\right)
$$ 
the associated maps.  Then we can take the system so that it satisfies
the following commutativity conditions hold
\be\label{eq:compatibility}
\pi_j \circ \pi_{j'}  = \pi_j,\quad  \rho_j \circ \pi_{j'} = \rho_j,
\ee
for all  pair $(j,j')$ satisfying $Y_j < Y_{j'}$, whenever the maps and compositions in \eqref{eq:compatibility} 
are defined.
\end{prop}
For the gluing purpose of the system of $L_\infty[1]$ structures on the system of neighborhoods $\{U_j\}$,
we add the following intersection property.

\begin{cond}\label{cond:triple-intersection}
We require 
$$
U_j \cap U_{j'} = \emptyset
$$
whenever $\dim Y_j = \dim Y_{j'}$. We call any such system of neighborhoods covering $M$
a \emph{good system of neighborhoods}.
\end{cond}
Such property can be ensured inductively starting from the minimal strata: This condition enables us to 
construct a compatible gluing of the $L_\infty$ structures 
inductively with respect to the partial order given by the dimension of the
stratum of the stratification.

\subsection{Dimension strata $Y_j$ and their null foliations}

We recall that each stratum $Y_j$ carries the presymplectic form which is nothing but
the restriction of the given two form $\omega$ on the total space $M$.

Consider the dimension strata $Y_j$. Then  $\omega$ has constant rank on each connected component of
$Y_j \setminus Y_{j-1}$ so that the component $Y_\alpha$ is naturally a presymplectic manifold equipped 
with the induced closed two form 
\be\label{eq:omegaj}
\omega_j = \iota_j^*\omega
\ee
restricted to the component. To simplify notations in the further discussion, 
we write the indexing set $(j,\ell)$ by the Greek letter $\alpha$ and denote the corresponding 
connected component by
$$
Y_j = Y_\alpha
$$
for $j = j(\alpha)$.
Then  we  further decompose each stratum $Y_j$ into its connected components so that
\begin{equation}
M = \bigcup\limits_{\alpha} Y_\alpha
\end{equation}
where $j = j(\alpha)$ is the dimension of $Y_\alpha$. We denote the nullity of $\omega_\alpha$ by
\be\label{eq:nullity-alpha}
m = m(\alpha).
\ee
We mention that $\nullity \omega_{\alpha} \geq \nullity \omega_{\alpha'}$
 whenever $j(\alpha) < j(\alpha')$ and $\alpha'$ is an ancestor of $\alpha$.

We now take a system of virtual symplectic 
neighborhoods $V_\alpha$ of $Y_\alpha$ which are the Gotay's normal forms $T^*\CF_\alpha$ 
\cite{gotay} over the presymplectic manifold  $Y_\alpha$.

\begin{defn}\label{defn:order}
Consider the indexing set of the connected components of dimension strata
$$
\mathfrak{P} := \{\alpha \mid 0 \leq  j \leq N - 2[N/2], \ Y_\alpha, \, j(\alpha) = j
\text{ is a connected component of } \, Y_j\}.
$$
We put a partial order on $\mathfrak{P}$ by $\alpha < \alpha'$  if and only if one of the following 
holds:
\begin{enumerate}
\item  $ j(\alpha) < j' (\alpha')$ or,
 \item $ j(\alpha) = j' (\alpha')$ and $m(\alpha) \geq m(\alpha')$.
\end{enumerate}
\end{defn}

We consider the system of tubular neighborhood $\{U_\alpha \supset Y_\alpha\}$ re-indexed by the connected components. Notice that by construction, $\ker \omega|_{Y_\alpha}$ defines a null foliation denoted by $\CF_\alpha$ on the stratum $Y_\alpha$. Then it follows that $(U_\alpha, \pi^*_{j(\alpha)} \omega|_{Y_\alpha})$ is a presymplectic manifold. We write 
$$
\widetilde \omega_\alpha := \pi^*_{j(\alpha)} \omega|_{Y_\alpha}
$$ 
and  denote by $\widetilde \CF_\alpha$ the associated null-foliation, i.e.,
$T\widetilde {\CF}_\alpha = \ker \widetilde {\omega}_\alpha$.

\begin{lem}\label{lem:dpijalpha-ker} We have 
$$
d\pi_{j(\alpha)} (\ker \widetilde \omega_{\alpha'}) \subset \ker \omega_\alpha.
$$
\end{lem}
\begin{proof} Let $p \in U_{\alpha}$ and $w \in T_p U_\alpha \cap \ker \widetilde \omega_\alpha|_p$.
Then there exists $Y_{\alpha'}$ such that $p \in U_\alpha \cap Y_{\alpha'}$ for some $\alpha'$ with
$j(\alpha) < j(\alpha')$.
By the compatibility of the system, we have 
$$
\pi_{j(\alpha)}(p) = \pi_{\alpha}^{\alpha'}(p) = x, \quad d\pi_{j(\alpha)}(w) = d\pi_{\alpha}^{\alpha'}(w).
$$
Let $v \in T_x Y_\alpha$ be an arbitrary vector. Since $\pi_\alpha^{\alpha'}$  is a submersion, we can
express $v = d\pi_\alpha^{\alpha'}(w')$ for some $w' \in T_p Y_{\alpha'}$.

We evaluate
\beastar
\omega_\alpha(d\pi_{j(\alpha)}w, v) & = & \pi^*_{j(\alpha)}\omega_\alpha(w, w') 
=  \pi^*_{j(\alpha)}\iota_{j(\alpha)}^* \omega(w, w')  \\
& = & ( \iota_{j(\alpha)} \pi_{j(\alpha)})^* \omega(w,w') = 
(\pi_{j(\alpha)})^*\omega(w,w') \\
& = & \omega(d\pi_{j(\alpha)}(w), d\pi_{j(\alpha)}(w')) 
= \omega( d\pi_{\alpha}^{\alpha'}(w),  d\pi_{\alpha}^{\alpha'}(w')) \\
& = & (\pi_\alpha^{\alpha'})^* \omega(w,w')  = \omega_{\alpha'}(w, w') = 0
\eeastar
where the last vanishing follows from the hypothesis $w \in \ker \omega_{\alpha'}$.
Therefore we have proved
$d\pi_{j(\alpha)}w \in \ker \omega_\alpha$ which finishes the proof.
\end{proof}

\section{Compatible choice of  polarizations $G_\alpha$}

In this section, we will make a \emph{compatible choice} of the system of 
polarizations $\{G_\alpha\}$.
The first matter of business for the purpose
is to describe the compatibility condition whose description is in order.

We start with a minimal strata $Y_{\text{\rm min}}$ and choose any slice $G_{\text{\rm min}}$ so that 
\be\label{eq:slice-minimal}
TY_{\text{\rm min}} = T\CF_{\text{\rm min}} \oplus G_{\text{\rm min}}.
\ee
Then we consider a pair of strata $Y_\alpha$ and $Y_{\alpha'}$ that is consecutive with
$\alpha< \alpha'$ with $\alpha = (j,m)$, $\alpha' = (j',m')$.
 By the semi-continuity of the nullity and the equality
$N = m (\alpha) + 2\ell(\alpha)$ for all $\alpha$, we have
$$
j \leq j'  \quad \text{\rm and} \quad  \ell \leq \ell'.
$$
By dimension counting, we have
$$
\dim G_\alpha = N - m(\alpha) = 2 \ell(\alpha).
$$
Denote by $V_{\alpha\alpha'}$ the intersection 
$$
U_\alpha \cap Y_{\alpha'} =: V_{\alpha\alpha'} \subset Y_{\alpha'}
$$
which is a tubular neighborhood of $Y_\alpha$ in $Y_{\alpha'}$. We denote by
$\pi_\alpha^{\alpha'}: Y_{\alpha'} \to Y_\alpha$ the tubular projection of 
$Y_{\alpha'}$ to $Y_\alpha$. Since 
$$
\omega_\alpha =\iota_\alpha^* \omega, \quad \omega_{\alpha'} = \iota_{\alpha'}^*\omega,
$$
Lemma \ref{lem:dpijalpha-ker} implies
$$
\ker  (\pi_{\alpha}^{\alpha'})^*\omega_\alpha \supset \ker \omega|_{Y_{\alpha'} }
$$
on $V_\alpha$. 

We will choose the collection of polarizations $\{G_{\alpha}\}_{\alpha \in \mathfrak{P}}$ so that the following holds.  
\begin{prop}\label{prop:Galpha'-Galpha} Let $\alpha < \alpha'$ be as above and assume $G_{\alpha}$
be given. Then we can choose a polarization $G_{\alpha'}$ of $Y_\alpha$ and a tubular neighborhood $V_\alpha \supset Y_\alpha$
so that
\be\label{eq:Galpha'-Galpha}
(\pi_{\alpha}^{\alpha'})^* G_{\alpha} \supset G_{\alpha'} \cap (\pi_\alpha^{\alpha'})^*TY_{\alpha}
\ee
holds on $V_\alpha$.
\end{prop}
\begin{proof} On $Y_{\alpha}$, we have
$$
T\CF_{\alpha} \subset TY_{\alpha} \subset (\iota_{\alpha})^*TM = TM|_{Y_{\alpha}}.
$$
Let $G_{\alpha}$ be the given splitting 
\be\label{eq:TMalpha}
TY_{\alpha} = G_{\alpha} \oplus T\CF_{\alpha}.
\ee
and consider the stratum $Y_{\alpha'}$. We will choose $G_{\alpha'}$ so that
\be\label{eq:splitting-alpha'}
TY_{\alpha'} = G_{\alpha'} \oplus T\CF_{\alpha'}
\ee
and that \eqref{eq:Galpha'-Galpha} holds. We recall
\be\label{eq:keralpha<keralpha'}
\ker \omega_\alpha = \ker (\iota_\alpha^{\alpha'})^*\omega_{\alpha'} \quad \text{\rm on } \, Y_\alpha
\ee
by the definition of $\omega_\alpha = \iota_\alpha^*\omega$ and 
$\iota_\alpha =\iota_{\alpha'}\circ  \iota_{\alpha}^{\alpha'}$.

\begin{sublem}\label{sublem:include} For any $\alpha < \alpha'$, the presymplectic form
$$
(\pi_\alpha^{\alpha'})^*\omega_\alpha 
$$
 satisfies
\be\label{eq:keralpha<keralpha'-pi}
\ker (\pi_\alpha^{\alpha'})^*\omega_\alpha \supset \ker \omega_{\alpha'} \quad \text{\rm on } \, Y_{\alpha'}.
\ee
\end{sublem}
\begin{proof} Let $v' \in TY_{\alpha'}$ satisfy $v' \intprod  \omega_{\alpha'} = 0$.
Recall
$$
TY_{\alpha'} = G_{\alpha'} \oplus T\CF_{\alpha'}
$$
and $G_{\alpha'}|_{Y_\alpha} \subset G_\alpha$ and that
$$
d\pi_\alpha^{\alpha'}(G_{\alpha'}) \subset G_\alpha.
$$
We evaluate
$$
(\pi_\alpha^{\alpha'})^*\omega_\alpha(v',w') 
= \omega_\alpha\left(d\pi_\alpha^{\alpha'}(v'), d\pi_\alpha^{\alpha'}(w')\right).
$$
By the compatibility $(\pi_\alpha, \rho_\alpha)$, we have
$$
d\pi_\alpha^{\alpha'}( T\CF_{\alpha'}) \subset  T\CF_{\alpha}.
$$
This implies $\omega_\alpha(d\pi_\alpha^{\alpha'}(v'), d\pi_\alpha^{\alpha'}(w')) = 0$
which in turn implies
$$
(\pi_\alpha^{\alpha'})^*\omega_\alpha(v',w') = 0
$$
for all $w' \in TY_{\alpha'}$. This proves $v' \in \ker (\pi_\alpha^{\alpha'})^*\omega_\alpha$,
which finishes the proof.
\end{proof}


We have the natural exact sequence
$$
0 \longrightarrow  (\pi_\alpha^{\alpha'})^* TY_\alpha  \longrightarrow 
TV_\alpha  \longrightarrow T_{Y_\alpha}V_\alpha \longrightarrow 0
$$
where $T_{Y_\alpha}V_\alpha$ is the normal bundle of $Y_\alpha$ in $V_\alpha$.
Pulling back \eqref{eq:TMalpha} by $\pi_\alpha^{\alpha'}$, we obtain
\be\label{pi*TMalpha}
 (\pi_\alpha^{\alpha'})^*TY_\alpha =  (\pi_\alpha^{\alpha'})^*(G_\alpha) \oplus  (\pi_\alpha^{\alpha'})^*(T\CF_\alpha)  .
\ee
Because of \eqref{eq:keralpha<keralpha'-pi},  $\omega_{\alpha'}$ is
nondegenerate on  $(\pi_\alpha^{\alpha'})^*G_{\alpha}$.

We take $G_{\alpha'}$ on $Y_{\alpha'}$  so that
\be\label{eq:Galpha'toGalpha}
G_\alpha' = (\pi_\alpha^{\alpha'})^*G_\alpha  \oplus N(T\CF_{\alpha'}; (\pi_\alpha^{\alpha'})^*T\CF_{\alpha})
 \ee
where $N(T\CF_{\alpha'}; (\pi_\alpha^{\alpha'})^*T\CF_{\alpha})$ is any geometric representative of the 
quotient bundle 
$$
(\pi_{\alpha}^{\alpha'})^*T\CF_{\alpha} /T\CF_{\alpha'}
$$
on $V_\alpha \cap Y_{\alpha'}$. 

\begin{lem} We have
$$
\ker \omega_{\alpha'} \cap (\pi_\alpha^{\alpha'})^*G_\alpha = \{0\}. 
$$
\end{lem}
\begin{proof}
Let $w \in (\pi_\alpha^{\alpha'})^*G_\alpha$ be a nonzero vector.
Then we have
$$
d\pi_\alpha^{\alpha'}(w) \in G_\alpha
$$
and hence there is a vector $v \in G_\alpha$ such that $\omega_\alpha(d\pi_\alpha^{\alpha'}(w), v) \neq 0$
by nondegeneracy of $\omega_\alpha$ on $G_\alpha$.
Since $d\pi_\alpha^{\alpha'}|_{G_\alpha'}$ is surjective onto $G_\alpha$, we can write
$$
v = d\pi_\alpha^{\alpha'}(w') 
$$
for some $w' \in G_\alpha'$, and hence
$$
0 \neq \omega_\alpha\left(d\pi_\alpha^{\alpha'}(w), v\right)
= \omega_\alpha\left(d \pi_\alpha^{\alpha'}(w), d \pi_\alpha^{\alpha'}(w') \right)
= (\pi_\alpha^{\alpha'})^*\omega_\alpha(w,w')
$$
This proves that $w \not \in \ker (\pi_\alpha^{\alpha'})^*\omega_\alpha$ which in turn
implies $w \not \in \ker \omega_{\alpha'}$ by Sublemma \ref{sublem:include}.
This finishes the proof.
\end{proof}

By definition of $\ker \omega_{\alpha'}$,  we also have
$$
\ker \omega_{\alpha'} \cap N \left(T\CF_{\alpha'}; (\pi_\alpha^{\alpha'})^*T\CF_{\alpha}\right)  = \{0\}.
$$
Combining the two,  we have shown that  $\ker \omega_{\alpha'} \cap G_{\alpha'} = 0$.
This proves that $G_{\alpha'}$ is a polarization of $\omega_{\alpha'}$ on $V_\alpha$.

Finally, to make the inclusion \eqref{eq:Galpha'-Galpha} hold, we
need to take  the geometric representative $N(T\CF_{\alpha'}; T\CF_{\alpha})$ suitably.
For this, we  just take
$$
N\left(T\CF_{\alpha'}; (\pi_\alpha^{\alpha'})^*T\CF_{\alpha}\right) := \ker \omega_{\alpha'} + \ker \pi_{\alpha}^{\alpha'}.
$$
Then
\beastar
G_{\alpha'} \cap (\pi_\alpha^{\alpha'})^*TY_{\alpha} & = &
\left((\pi_\alpha^{\alpha'})^*G_\alpha + \ker \omega_{\alpha'} + \ker \pi_{\alpha}^{\alpha'}\right)
\cap (\pi_\alpha^{\alpha'})^*TY_{\alpha} \\
&\subset& (\pi_\alpha^{\alpha'})^*G_\alpha \cap (\pi_\alpha^{\alpha'})^*TY_{\alpha}
\subset (\pi_\alpha^{\alpha'})^*G_\alpha.
\eeastar  
This finishes the proof of the proposition.
\end{proof}
The equation \eqref{eq:Galpha'toGalpha} also shows the way how we extend 
the given $G_{\alpha}$ to the slice on $Y_{\alpha'}$ and $V_{\alpha'}$: We have only to 
choose the normal bundle \emph{inside the stratum $Y_{\alpha'}$ without affecting
the given polarization $G_{\alpha}$ on $Y_{\alpha}$ but just restricting it to $Y_{\alpha'}$.}
This enables us to perform the inductive construction of $G_\alpha$ till we reach 
the maximal strata by choosing the system of tubular neighborhoods $\{(U_\alpha, \pi_\alpha,\rho_\alpha)\}$ 
small enough.

We write
$$
 \left(V_\alpha, \widetilde \omega_\alpha,\widetilde G_\alpha\right)
 : =\left (V_\alpha, (\pi_\alpha^{\alpha'})^*\omega_\alpha,  (\pi_\alpha^{\alpha'})^*G_\alpha\right)
$$
and call it a \emph{pull-back presymplectic tubular neighborhood} of the polarized symplectic stratum
$(Y_\alpha,\omega_\alpha,G_\alpha)$. We also  simply denote it by
\be\label{eq:T-alpha}\
\CV_\alpha^\omega
 : = \left(V_\alpha, \widetilde \omega_\alpha,\widetilde G_\alpha\right) \supset (Y_\alpha,\omega_\alpha,G_\alpha)
 \ee

\begin{defn} We call $Y_\alpha$ a presymplectic stratum of $\omega$ and the system of  triples
$$
\CP\CS\omega: = \{(Y_\alpha,\omega_\alpha, G_\alpha)\}_{\alpha \in \mathfrak P}
$$
a \emph{polarized presymplectic stratification}, and call the system
$
\{ (V_\alpha, \widetilde \omega_\alpha,\widetilde G_\alpha)\}_{\alpha \in \mathfrak P}
$
the  \emph{pull-back stratified presymplectic atlas} of $(M,\omega)$
\end{defn}

\section{Compatible system of $L_\infty$ spaces over the presymplectic strata}
\label{sec:compatible-system-linftyspaces}
Let 
$$
\{\CV_\alpha^\omega\} = \left\{(V_\alpha, \widetilde \omega_\alpha, \widetilde G_\alpha)\right\}_{\alpha \in \mathfrak P}
$$
 be the pull-back presymplectic atlas of the polarized presymplectic stratification
$\CP\CS{(M,\omega)}=\{(Y_\alpha, \omega_\alpha, G_\alpha)\}_{\alpha \in \mathfrak P}$
of $(M,\omega)$. We denote by
$$
\mathfrak l^{\CV^\omega}_\alpha = \left\{\mathfrak l^{\CV^\omega_\alpha}_k \right\}_{k \geq 0}
$$
the $L_\infty[1]$ algebra  associated to 
$\CV_\alpha^\omega = (V_\alpha, \widetilde \omega_\alpha, \widetilde G_\alpha)$.

\subsection{Statement of compatibility}
\label{subsec:ellinfty-compatibility}

We start with applying the main theorem of \cite{oh-park:coisotropic}, Theorem \ref{thm:oh-park-intro}, to
each $\CV_\alpha^\omega$:
 For each $\CV_\alpha^\omega$, we can 
canonically equip it with an $L_\infty[1]$-structure on the graded complex
$$
\left( \bigoplus_{*} \Omega^*(\CF_\alpha)[1], \{\mathfrak l_k^\alpha\}_{\ell \geq 1}\right).
$$
We denote by $\mathfrak l^{\CV_\alpha^\omega}[1]$ the corresponding  $L_\infty[1]$ algebra.

The following theorem is also proved (\emph{in the formal level}) in \cite[Theorem 10.1]{oh-park:coisotropic}
using the fact that \emph{the choice of splittings is contractible.}
(See \cite[Definition 8.3]{fukaya:cyclic}, \cite[Definition 3.35]{fukaya:immersed} for the definition of 
pseud-isootopy in the context of $A_\infty$ algebra which can be easily adapted to the case of $L_\infty$ algebras.)

\begin{thm}[Theorem 10.1 \cite{oh-park:coisotropic}] Let $(M,\omega)$ be a presymplectic manifold and
consider two splittings 
$$
TM = G \oplus T\CF_\omega = G' \oplus T\CF_\omega.
$$
Denote by $\mathfrak l^G$ and $\mathfrak l^{G'}$ be the
associated $L_\infty[1]$ algebras. Then the two are canonically $L_\infty$ pseudo-isotopic.
\end{thm}
Many variations of this theorem have been proved both in the formal level and in the analytic
level in various more general context in the more systematic way \emph{utilizing Voronov's $V$-data
formalism}, but always with the null foliations of \emph{constant nullity}. 
(See \cite{cattaneo-schatz}, \cite{LOTV}, for example.) We will also need the 
corresponding result for the case of $L_\infty$ algebras for which we refer readers to
and \cite[Section 14]{fukaya:cyclic}, \cite[Defintion 3.35]{fukaya:immersed}
for the defintition of pseudo-isotopies, which can be easily adapted to the case of 
$L_\infty$ algebras by symmetrization.)

However, the current situation deals with a new more general 
situation different therefrom which considers geometric context  in that
we consider two presymplectic strata $(Y_\alpha, \omega_\alpha)$, $(Y_{\alpha'},\omega_{\alpha'})$
of \emph{different nullities} but with the compatibility relation
$$
\ker \left((\pi_\alpha^{\alpha'})^*(\omega_\alpha) \right) \supset \ker \omega_{\alpha'}.
$$
\begin{thm}\label{thm:gluing-morphisms}
Let $\{\CV_\alpha^\omega\}_{\alpha \in \mathfrak P}$ be a 
be a pull-back presymplectic atlas of the polarized presymplectic stratification
$$
\CP\CS{(M,\omega)}=\{(Y_\alpha, \omega_\alpha, G_\alpha)\}_{\alpha \in \mathfrak P}
$$
of $(M,\omega)$. Then the following holds:
\begin{enumerate}
\item For each consecutive pair $(\alpha,\alpha')$ with $\alpha < \alpha'$.
Then there exists an $L_\infty$ morphism 
$$
\mathfrak f^{\alpha\alpha'}: \mathfrak l^{\CV^\omega}_\alpha \to \mathfrak l^{\CV^\omega}_{\alpha'}.
$$
\item For each triple $(\alpha,\beta,\gamma)$ with $\alpha < \beta < \gamma$, the composition 
$\mathfrak f^{\beta\gamma} \circ \mathfrak f^{\alpha\beta}$ are defined
on an open subset $V_{\alpha\beta\gamma} \subset V_\alpha \cap V_\beta \cap V_\gamma$.
\item The two $L_\infty$ morphisms 
$\mathfrak f^{\beta\gamma} \circ \mathfrak f^{\alpha\beta}$ and $\mathfrak f^{\alpha\gamma}$
are canonically $L_\infty$ pseudo-isotopic.
\end{enumerate}
\end{thm}
The proof of this theorem will occupy  the next subsection, and the whole of Section \ref{sec:gluing}.

\subsection{Construction of $L_\infty$ homomorphisms and their composition}

For the proof of Statement (1) of Theorem , we start with the splitting
$$
TY_\alpha = G_\alpha \oplus T\CF_\alpha, \quad TY_{\alpha'} = G_{\alpha'} \oplus T\CF_{\alpha'}
$$
with the inclusion relations
$$
(\pi_\alpha^{\alpha'})^*T\CF_\alpha \supset T\CF_{\alpha'}, \quad (\pi_\alpha^{\alpha'})^*(G_\alpha) \subset G_{\alpha'}
$$
for the choice of $G_{\alpha'}$ given in  \eqref{eq:Galpha'toGalpha}.

We denote by $R_\rho^t$ the deformation retraction obtained by the gradient flow of the 
function $\rho$ in the system of tubular neighborhoods $(T_j,\pi_j, \rho_j)$.
Then by the quadratic property of the tubular function $\rho$  the family
$R_\rho^t$ is smooth on $[0,1]$ and the 
 compatibility condition \eqref{eq:compatibility} implies 
 $R_\rho^1 = \pi_\alpha^{\alpha'}$. It also implies that the flow $R_\rho^t$ is tangent to the fiber of
  the projection $\pi_\alpha^{\alpha'}$.

We then consider 
$$
\omega_t = \begin{cases} (R_\rho^t)_*\omega_{\alpha'}, \quad &  t \in (0,1]\\
(\pi_{\alpha}^{\alpha'})^*\omega_\alpha \quad & t=0.
\end{cases}
$$
where we have $R_\rho^0 = \id|_{Y_{\alpha'}}$ and $R_{\rho}^1 = \pi_{\alpha}^{\alpha'}$.
We also have $\omega_0 = \omega_{\alpha'}$ and 
$$
\omega_1 = (\pi_\alpha^{\alpha'})^*\omega_\alpha:
$$
This retraction $R_\rho^t$ is generated by the \emph{smooth vector field}
\be\label{eq:Yt}
Y_t = \begin{cases}
\frac{\del R_\rho^t}{\del t}\circ (R_\rho^t)^{-1} \quad  & t > 0 \\
0 \quad & t = 0
\end{cases}
\ee
which is smooth at $t = 0$ since $\rho$ satisfies $\rho(v) = \frac12\|v\|^2$ near $v = 0$
as given in \eqref{eq:rho}.

Since $\ker \widetilde \omega_\alpha \supset \ker \omega_{\alpha'}$, we have
$$
\ker \omega_t \supset \ker \omega_{\alpha'}
$$
for all $t \in [0,1]$ and $\ker \omega_t = \ker \omega_{\alpha'}$ for $t \in (0, 1]$.
In particular, the splitting \eqref{eq:Galpha'toGalpha}
$$
G_{\alpha'} = (\pi_\alpha^{\alpha'})^*TY_\alpha  \oplus N\left(T\CF_{\alpha'};(\pi_\alpha^{\alpha'})^*T\CF_{\alpha}\right)
$$
is invariant under the flow of $R_\rho^t$ on $V_\alpha \cap V_{\alpha'}$. 

Then we will apply Moser's stability theorem. For this purpose, as usual we
set and differentiate the equation $(\phi^t)^*\omega_t = \omega_{\alpha'}$ 
which we solve by integrating its generating vector field
$$
X_t = \frac{\del \phi^t}{\del t} \circ (\phi^t)^{-1}.
$$
This vector field is determined by the equation
$$
X_t \intprod \omega_t + \frac{d \omega_t}{d t} = 0
$$
which is obtained by differentiating the equation $(\phi^t)^*\omega_t = \omega_{\alpha'}$ in $t$.
On the other hand, we have
$$
\frac{d \omega_t}{d t} =\frac{d}{dt} (R_\rho^t)^*\omega_{\alpha'} 
= (R_\rho^t)^*d\left(Y_t \intprod \omega_{\alpha'}\right)
= d\left((R_\rho^t)^*\left(Y_t \intprod \omega_{\alpha'}\right)\right)
$$
and hence
\be\label{eq:omegaalpha'-alpha}
\omega_t =  \omega_{\alpha'} + d \beta_t
\ee
with the one-form
$$
\beta_t = \int_0^t (R_\rho^u)^*\left(Y_u \intprod \omega_{\alpha'}\right) \, du.
$$
Therefore we obtain
\be\label{eq:Moser-equation}
X_t \intprod \omega_t = - \frac{d \omega_t}{dt} = - \beta_t.
\ee
Since $\ker \omega_t = \ker (\pi_\alpha^{\alpha'})^*\omega_\alpha$ for $t > 0$
and $\ker \omega_0 = \ker \omega_{\alpha'}$, we may put
$$
X_t \in G_{\alpha'} = (\pi_\alpha^{\alpha'})^*TY_\alpha  \oplus 
N\left(T\CF_{\alpha'};(\pi_\alpha^{\alpha'})^*T\CF_{\alpha}\right).
$$
Therefore we can uniquely solve \eqref{eq:Moser-equation} provided
$$
\beta_t\big|_{(\pi_\alpha^{\alpha'})^*N(Y_\alpha;Y_{\alpha'})} = 0:
$$
 We note that $Y_t$ is tangent to the fiber of $\pi_\alpha^{\alpha'}$  and hence $Y_t \intprod \omega_{\alpha'} = 0$ on $\ker \omega_{\alpha'} \supset (\pi_\alpha^{\alpha'})^*N(Y_\alpha;Y_{\alpha'})$.
Since $R_\rho^t$ is generated by $Y_t$, the same vanishing holds for $\beta_t$,
 which concludes that we can
solve the equation \eqref{eq:Moser-equation}. \emph{Finally we apply the time-reversal $t \mapsto 1-t$}.

In summary, we have achieved the identity \eqref{eq:omegaalpha'-alpha} with the 
above described properties and \emph{the time $t$ replaced by $1-t$}. 
Then Statement (1) follows from the following general result,
whose proof we postpone till the next section.
\begin{prop}\label{prop:Linfty-morphism}
Let $\omega_1 = \omega_0 + d\beta$ with $\ker \omega_1 \supset \ker \omega_0$ and $G_{\alpha'}$ is
invariant under the flow $ (R_t^\rho)$.
 Then there is an $L_\infty$ homomorphism
$$
\mathfrak f^{10}: \mathfrak l^{\CV^{\omega_1}} \to \mathfrak l^{\CV^{\omega_0}}.
$$
\end{prop}

Statement (2) of Theorem \ref{thm:gluing-morphisms}
 immediately follows since on $V_\alpha \cap V_\beta \cap V_\gamma$ 
all three morphisms $f^{\alpha\beta}, \, f^{\beta\gamma}$ and $f^{\alpha\gamma}$ are 
defined.

For Statement (3), we will again utilize Proposition \ref{prop:Linfty-morphism}
by connecting the composition 
$\mathfrak f^{\beta\gamma} \circ \mathfrak f^{\alpha\beta}$ and $\mathfrak f^{\alpha\gamma}$
by a one-parameter family $\mathfrak f_t$ that satisfies
\be\label{eq:bdy-condition}
\mathfrak f_0 = \mathfrak f^{\beta\gamma} \circ \mathfrak f^{\alpha\beta}, \quad
\mathfrak f_1 = \mathfrak f^{\alpha\gamma}.
\ee
Note that in this case, both morphisms have the common domain and the codomain.

\section{Gluing of $L_\infty$ structures into a stratified $L_\infty$ space}
\label{sec:gluing}

To explicitly describe the composition $\mathfrak f^{\beta\gamma} \circ \mathfrak f^{\alpha\beta}$
in terms of the underlying presymplectic geometry, we will utilize Cattaneo-Sch\"atz 
description of $L_\infty[1]$ algebras in terms of the $V$ algebra data. Then we closely follow 
the explanation of \cite[Section 3.5]{LOTV} on how the geometric 
deformation of presymplectic forms $\omega_t$ induces the associated family $\mathfrak f_t$
of $L_\infty$ morphisms, but \emph{allowing the circumstance 
where the nullity drops down  (or equivalently the rank jumps up) at $t = 0$.}

\subsection{Construction of gauge-transformation}

We recall the exposition of \cite{cattaneo-schatz} and \cite{LOTV} here.
We start with recalling the definition of the \emph{$V$-data}, following
the terminology introduced in \cite[Section 1.2, Lemma 2.2]{fregier-zambon} 
and \cite{cattaneo-schatz}
which summarizes Voronov's construction. (See \cite{voronov1,voronov2}.
We alert readers that \cite{fregier-zambon} denotes our $\pi$ here by $P$.)

\begin{defn}[$V$-data]
Consider a triple $(\mathfrak{h}, \mathfrak{a}, \pi, \Delta)$ such that
\begin{enumerate}
\item $\mathfrak{h}$ is a graded Lie algebra over a field $\mathbb{K}$,
\item $\mathfrak{a}$ is an abelian subalgebra of $\mathfrak{h}$,
\item We have a splitting $\mathfrak h \simeq \mathfrak{a} \oplus \mathfrak n$,
\item $\pi : \mathfrak{h} \rightarrow \mathfrak{a}$ is the associated projection.
\end{enumerate}
An element  $\Delta  \in \mathfrak{h}$ is called a \emph{Maurer-Cartan element}
if it has degree 1 and satisfies  $[\Delta,\Delta]=0$. 
The quadruple $(\mathfrak{h}, \mathfrak{a}, \pi, \Delta)$ 
is called a \textit{V-data}. 
\end{defn}

We denote by $i: \mathfrak a \to \mathfrak h$ and $\pi: \mathfrak h \to \mathfrak a$ the inclusion and
the projection operators.
Now we define a family of operators $l^{P}_k : \mathfrak{a}^{\otimes k} \rightarrow \mathfrak{a}$ defined by
\bea
(x_1, \cdots x_k) &\mapsto&  \pi[ \cdots [[\Delta, i(x_1)], i(x_2)], \cdots, i(x_k)], \quad \text{if }k \geq 1,\label{eq:k>1}\\
1 &\mapsto&  \pi \Delta,  \quad \quad \quad \quad \quad \quad \quad \quad \quad \quad \text{if } k = 0. \label{eq:k=1}
\eea
\begin{lem}\label{valinf}
The family $\{l^\Delta_k\}_{k \geq 0}$ forms a curved $L_{\infty}[1]$-algebra.
\end{lem}
\begin{proof}
Observe that the Jacobiator for each $n$ is given by $l^{\frac{1}{2}[\Delta,\Delta]}_n \equiv 0.$
\end{proof}
 Then Voronov's construction associates 
a curvature-free $L_\infty[1]$-algebra to each such $V$-data.

Then we have the following result.
\begin{thm}[Theorem 3.2 \cite{cattaneo-schatz}, \, Theorem 3.19. \cite{LOTV}]\label{theor:CS}
Let $(\mathfrak h,\mathfrak a, \pi, \Delta_0)$ and $(\mathfrak h, \mathfrak a, \pi, \Delta_1)$ be $V$-data, and let $\mathfrak a_0$ and $\mathfrak a_1$ be the associated $L_\infty[1]$-algebras. If $\Delta_0$ and $\Delta_1$ are gauge equivalent and they are intertwined by a gauge transformation preserving $\ker \pi$, then $\mathfrak a_0$ and $\mathfrak a_1$ are $L_\infty$-isomorphic.
\end{thm}

Cattaneo and Sch\"atz's main idea of the proof of this theorem is to prove that when
\begin{itemize}
\item $\Delta_0$ and $\Delta_1$ are gauge equivalent elements of the graded Lie algebra $\mathfrak h$, and
\item they are intertwined by a gauge transformation preserving $\ker \pi$,
\end{itemize}
then $\mathfrak a_0$ and $\mathfrak a_1$ are $L_\infty$-isomorphic. Specifically, they showed that
 $\Delta_0$ and $\Delta_1$ are \emph{gauge equivalent} if they are interpolated by a smooth family $\{ \Delta_t \}_{t \in [0,1]}$ of elements $\Delta_t \in \mathfrak h$, and there exists a smooth family $\{ \xi_t \}_{t \in [0,1]}$ of degree zero elements $\xi_t \in \mathfrak h$ such that 
 $\Delta_t$ satisfies
 \be\label{eq:gaugeMC00}
\frac{d}{dt} \Delta_t = [\xi_t , \Delta_t].
\ee
One usually assumes that the family $\{ \xi_t \}_{t \in [0,1]}$ integrates to a family $\{ \phi_t \}_{t \in [0,1]}$ of automorphisms $\phi_t : \mathfrak h \to \mathfrak h$ of the Lie algebra $\mathfrak h$, i.e.~$\{ \phi_t \}_{t \in [0,1]}$ is a solution of the Cauchy problem
\begin{equation}
\label{eq:Cauchy}
\left\{
\begin{aligned}
& \frac{d}{dt}\phi_t (-)=[\phi_ t (-), \xi_t] \\
& \phi_0= \id
\end{aligned}
\right. .
\end{equation}

Finally we say that $\Delta_0$ and $\Delta_1$ are intertwined by a gauge transformation preserving $\ker \Pi$ if family $\{ \xi_t \}_{t \in [0,1]}$ above satisfies the following conditions:
\begin{enumerate}
\item the only solution $\{ a_t \}_{t \in [0,1]}$, where $a_t \in \mathfrak a$, of the Cauchy problem
\begin{equation}
\label{eq:equivalence_2a}
\left\{
\begin{aligned}
& \frac{d}{dt}a_t =\pi [a_t , \xi_t] \\
& a_0 =0
\end{aligned}
\right.
\end{equation}
is the trivial one: $a_t = 0$ for all $t \in [0,1]$,
\item $[ \xi_t, \ker \Pi_t] \subset \ker \Pi_t$ for all $t \in [0,1]$.
\end{enumerate}

Now  unlike \cite{LOTV} or \cite{voronov1}, we allow to vary all arguments of the quadruple 
for $i = 0, \, 1$ where they vary only the derivation $\Delta$ fixing the rest. 
In that context, we will apply Cattaneo-Sch\"atz's construction after some modification.

We go back to the proof of Statement (3) of Theorem \ref{thm:gluing-morphisms}. 
To explicitly describe the composition $\mathfrak f^{\beta\gamma} \circ \mathfrak f^{\alpha\beta}$
in terms of the underlying presymplectic geometry, we will utilize Cattaneo-Sch\"atz's
description of $L_\infty[1]$ algebras in terms of the $V$ data aforementioned in the beginning of
the present subsection. Then we closely follow 
the explanation of \cite[Section 3.5]{LOTV} on how the geometric 
deformation of presymplectic forms $\omega_t$ induces the associated family $\mathfrak f_t$
of $L_\infty$ morphisms, but \emph{allowing the circumstance 
where the nullity drops down  (or equivalently the rank jumps up) at $t = 0$.} 

\subsection{Bivector field $Z_\alpha$ associated to presymplectic form $\omega_\alpha$}

 Let $E_i = T\CF_i$ associated to the presymplectic
strata $Y_0$ and $Y_1$. We will then consider
 two first order differential operators  $\Delta_i: \Gamma(E_i) \to \Gamma(E_i)$
  for $i = 0, \, 1$ arising from certain bivector fields $Z_\alpha$ associated to 
polarized presymplectic manifolds $(Y_\alpha,\omega_\alpha,G_\alpha)$.

We start with the splittings
$$
TY_\alpha = G_\alpha \oplus T\CF_\alpha.
$$
This splitting also induces the splitting
\be\label{eq:Lambda2TM}
\Lambda^2(TY_\alpha) = \Lambda^2(G_\alpha) \oplus (G \otimes T\CF_\alpha)
\oplus (T\CF_\alpha \otimes G) \oplus \Lambda^2(T\CF_\alpha)
\ee
and similar decomposition holds for $\Lambda^2(T^*Y_\alpha)$.

Recalling $\omega_\alpha$ is nondegenerate on $G_\alpha$, 
this decomposition induces an isomorphism
$$
\Lambda^2(G_\alpha^*) \to \Lambda^2(G_\alpha).
$$
In particular, we denote by 
$$
(\omega_\alpha|_{G_\alpha})^{-1} \in  \Lambda^2(G_\alpha) 
$$
the bivector field associated to $\omega|_{G_\alpha} \in \Lambda^2(G_\alpha^*)$. Now we are ready to
write down $Z_\alpha^\#$ explicitly as follows.

\begin{defn}[Bivector field $Z_\alpha$] We first define a bivector field
\be\label{eq:Zsharpalpha}
Z^\sharp_\alpha = (\omega|_{G_\alpha})^{-1} \oplus 0 \oplus 0 \oplus 0
\ee
on $Y_\alpha$ according to the decomposition \eqref{eq:Lambda2TM}. 
We then further extend the bivector field $Z^\sharp_\alpha$ to 
one $Z_\alpha$ defined on the tubular neighborhood  $V_\alpha \subset M$ as follows:
We take any splitting
\be\label{eq:splitting-alpha}
TM|_{V_\alpha} = (\pi_\alpha)^*(TY_\alpha) \oplus  (\pi_\alpha)^*(N(Y_\alpha;M))
\ee
and similarly  put
\be\label{eq:Zalpha}
Z_\alpha(x) =  \left(d_x\pi_\alpha|_{ (\pi_\alpha)^*(TY_\alpha)} \right)^{-1}(Z^\sharp_\alpha(\pi_\alpha(x))) \oplus 0
\oplus 0 \oplus 0
\ee
according to the decomposition of $\Lambda^2(TV_\alpha)$ induced by \eqref{eq:splitting-alpha}.
Here  $N(Y_\alpha;M)$ is any geometric representative of the normal bundle 
$T_{Y_\alpha}M$ of $Y_\alpha$ in $M$.
\end{defn}

We apply the above construction to the two families of two-forms
$$
\omega_t^{\alpha\beta}= (R_{\rho_\alpha}^t)^*\omega_\beta, \quad 
\omega_t^{\beta\gamma}= (R_{\rho_\beta}^t)^*\omega_\gamma
$$
for each triple $\alpha < \beta < \gamma$ on
the triple intersection $V_\alpha \cap V_\beta \cap V_\gamma$.

\begin{lem}\label{lem:deform-omega} We can inductively deform Mather's compatible system $(\pi_\alpha,\rho_\alpha)$ so that
the flow $R_{\rho_\alpha}^t$ respects the splitting \eqref{eq:splitting-alpha}.
In particular, we can express
$$
\omega_t^{\alpha\beta} = \begin{cases} 0 \quad & \text{\rm on } (\pi_\alpha)^*(TY_\alpha) \\
C_{\alpha\beta} e^{c_{\alpha\beta} t} \omega_\beta \quad & \text{\rm on }  (\pi_\alpha)^*(N(Y_\alpha;M))
\end{cases}
$$
for some constants $C_{\alpha\beta},\, c_{\alpha\beta} >0$ depending only on the pair $(\alpha,\beta)$.
\end{lem}
\begin{proof} To ensure the first requirement, it is enough to adjust $(\pi_\beta,\rho_\beta)$ so that
the fiber of $\pi_\alpha^\beta$ is tangent to $(\pi_\alpha)^*(N(Y_\alpha;M))$ and the gradient
vector field of $\rho_\alpha^\beta$ is tangent to $(\pi_\alpha)^*(N(Y_\alpha;M))$. This can be 
always done inductively.

To prove the second requirement, we first provide the description in the normal form of the pair $(\pi,\rho)$ 
in the direction of the \emph{normal direction}: It
is the radial contraction on $\R^{2k}$  generated by the radial function $f = \frac12 r^2$, and
the canonical symplectic form $\omega_0 = \sum_{i=1}^k dq_k \wedge dp_k$: The gradient vector field 
of $f$ is given by $r \frac{\del}{\del r}$ which satisfies
$$
\CL_{r\frac{\del}{\del r}} \omega_0 = \omega_0
$$
and hence $(R_f^t)^*\omega_0 = e^t \omega_0$. 

Applying this to the actual circumstance using the normal exponential map and further deforming $(\pi_\beta,\rho_\beta)$ 
if necessary, the gradient flow satisfies
$$
(R_{\rho_\beta}^t)^*\omega_\beta = C_{\alpha\beta} e^{c_{\alpha\beta} t} \omega_\beta
$$
for some choices of rescaling constants $C_{\alpha\beta}> 0$, $c_{\alpha\beta}> 0$.
This finishes the proof.
\end{proof}

We denote by $Z_t^{\alpha\beta}$ and $Z_t^{\beta\gamma}$ the families of the bivector fields associated
thereto respectively. Then, motivated by Lemma \ref{lem:deform-omega}, we consider the family of bivector fields
\be\label{eq:deform-Z}
Z_t^{\alpha\beta} =
\left(d_x\pi_{\alpha\beta}|_{ (\pi_{\alpha\beta})^*(TY_\alpha)} \right)^{-1}(Z^\sharp_{\alpha\beta}
(\pi_{\alpha\beta}(x))) \oplus 0
\oplus 0 \oplus 0
\ee
similarly as defined in \eqref{eq:Zalpha}.

\begin{lem} The bivector field $Z_{1/s}^{\alpha\beta}$ can be smoothly extended to $X_\alpha$ 
on $[0,1]$, but it cannot be smoothly extended across $s = 0$ to the negative interval $(-\epsilon, 0]$.
\end{lem}
\begin{proof} For the simplicity of exposition, we set $\rho = \rho_\beta$.
We estimate $(R_\rho^t)^*Z_t^{\alpha\beta}$ by evaluating it against $\omega_\beta$, i.e.,
$$
\omega_\beta\left((R_\rho^t)^*Z_t^{\alpha\beta}\right) = (R_\rho^{-t})^*\omega_\beta(Z_\beta)
 = \frac1C_{\alpha\beta} e^{-c_{\alpha\beta} t} \omega_\beta(Z_\beta).
$$
Therefore by making the change of variable $s = \frac1t$ on $(0,1]$, the transformed flow has the form
\be\label{eq:R-1/s}
(s,x) \mapsto R_\rho^{-1/s}(x)
\ee
and its generating vector field has the form
$$
Z'_{\alpha\beta}(s,x) = - \frac{1}{s^2} Z_\beta \left(\frac1s,x\right) =  - \frac1{C_{\alpha\beta}} \frac{1}{s^2} e^{-c_{\alpha\beta} 1/s} Z_\beta
$$
The lemma follows from the examination of this explicit formula and the property that the function 
$e^{-1/s}$ is smooth at $0$ from the positive side which however cannot be extended to the negative direction of
$s$.
\end{proof}

Therefore we can solve the Cauchy problem \eqref{eq:Cauchy}
\beastar
\begin{cases}
\frac{d}{dt} \xi_t =[\xi_t, Z_t] \\
\phi_0= \id
\end{cases}
\eeastar
for a degree zero element $\xi_t \in \mathfrak h$ in the positive direction which however is not solvable to the
negative direction. This completes the construction of the $L_\infty$ morphism $\mathfrak{f}^{\alpha\beta}$
and hence the proof of Proposition \ref{prop:Linfty-morphism}.

\begin{rem} 
\begin{enumerate}
\item It is worthwhile to separately mention that in our current geometric circumstance 
the above discussion gives rise to
a variation of \cite[Theorem 3.2]{cattaneo-schatz}, \cite[Theorem 3.19]{LOTV}:
Let $(\mathfrak h,\mathfrak a_0, \pi \Delta_0)$ and $(\mathfrak h, \mathfrak a_1, \pi \Delta_1)$ be two $V$-data 
of the $L_\infty[1]$-algebras associated to the strata $\alpha$ and $\beta$ as above. Then we have an $L_\infty$ morphism
from $\mathfrak a_0$ to $\mathfrak a_1$, \emph{which is not necessarily an $L_\infty$ quasi-isomorphism}. This is because the  Cauchy problem 
$$
\begin{cases}
 \frac{d}{dt}\phi_t (-)=[\phi_ t (-), \xi_t] \\
 \phi_0= \id
\end{cases}
$$
is solvable only towards the positive direction, not in the negative direction. We would like to
compare this with the solvability problem of the heat equation.
\item Indeed, it it is best to regard the isotopy  \eqref{eq:R-1/s} as a smooth map on the space-time
$M \times [0,1]$ \emph{ in the \emph{ admissible category of stratified manifolds} 
in the sense of \cite[Section 25.1]{fooo:book-kuranishi}}.
\end{enumerate}
\end{rem}

\subsection{Gluing of $L_\infty[1]$ structures into a stratified $L_\infty$ space}

Suppose that $\omega$ is a nice closed two-form on a smooth manifold $M$.

To each element $\mathfrak p = (j, \ell) \in \mathfrak{P}$,
we associate a polarization to $(U_{\mathfrak p},\omega_{\mathfrak p})$
and the associated  the foliation de-Rham complex 
 $\Omega^*(\CF_{\mathfrak p})$ and its $L_\infty$-enhancement given in Theorem \ref{thm:oh-park-intro}
$$
\left( \bigoplus_{*} \Omega^*(\CF_{\mathfrak p})[1], \{\mathfrak l_k^{\mathfrak p}\}_{k \geq 1}\right).
$$
We denote by ${\mathfrak l}_{\CU_{\mathfrak p}}^\omega$ the corresponding 
$L_\infty[1]$ algebra.

For $\mathfrak{p} := (j, \ell), \mathfrak{q} := (j', \ell') \in \mathfrak{P}$ with $\mathfrak{p} \leq \mathfrak{q},$ we consider the triple
\begin{equation}\nonumber
\left(U_{\mathfrak{p}\mathfrak{q}}, \phi_{\mathfrak{p}\mathfrak{q}}, \widehat{\phi}_{\mathfrak{p}\mathfrak{q}}\right)
\end{equation}
whose definitions are given as follows. 

Notice that $\mathfrak{p} \leq \mathfrak{q}$ implies $U_{\mathfrak{p}} \cap U_{\mathfrak{q}} \neq \emptyset.$ Then we denote
$$
{U}_{\mathfrak{p}\mathfrak{q}} := \psi_{\mathfrak{p}}^{-1} \left( \psi_{\mathfrak{p}}(\mathring{U}_{\mathfrak{p}})  \cap \psi_{\mathfrak{q}}(\mathring{U}_{\mathfrak{q}})\right)
$$
and define the map
$
\widehat \phi_{\mathfrak{p}\mathfrak{q}} : {\mathfrak l}_{\CU_{\mathfrak q}}^\omega
\to {\mathfrak l}_{\CU_{\mathfrak p}}^\omega
$
to be $\mathfrak{f}^{{\mathfrak q}{\mathfrak p}}$. We summarize the above discussion into
the following.

\begin{thm}\label{krnsthm}
$\big(M, \mathfrak{P}, \{\mathcal{U}_{\mathfrak{p}}\}, \{ \Phi_{\mathfrak{p}\mathfrak{q}} \} \big)$ is a 
good coordinate system (in the sense of \cite{fukaya-ono:arnold,fooo:book1} for the stratifies $L_\infty$ space associated to a generic
closed two-form $\omega$ on $M$. We call the resulting space a stratified $L_\infty$ space.
\end{thm}
\emph{We require that the system $\{\mathcal{U}_{\mathfrak{p}}\}$ of neighborhoods  be good
in the sense of Condition \ref{cond:triple-intersection}.}

\appendix

\section{Definition of Whitney stratifications}

In this appendix, we recall the definition of Whitney stratification \cite{thom}, \cite{goresky-macpherson}: 
Let $Z$ be a closed subset of 
 a smooth manifold $M$, and suppose that
 $$
 Z = \bigcup_{\alpha \in \CS} S_\alpha
 $$
 is an $\CS$-decomposition of $Z$, where $\CS$ is a partially ordered set. This decomposition is a Whitney
 stratification of $Z$ provided:
 \begin{enumerate}
 \item Each piece $S_\alpha$ is a locally closed smooth submanifold of $M$.
 \item Whenever $S_\alpha < S_\beta$ then the pair satisfies Whitney's conditions $A$ and $B$:
 suppose $x_i \in S_\beta$ is a sequence of points converging to some $y \in S_\alpha$. Suppose
 $y_i \in S_\alpha$ also converges to $y$, and suppose that (with respect to some local coordinate system on
 $M$) the secant lines $\ell_i = \overline{x_iy_i}$ converges to some limiting line $\ell$, and the tangent planes
 $T_{x_i}S_\beta$ converge to some limiting plane $\tau$. Then we have
 \begin{enumerate}
 \item (\text{\rm Condition A})  $ T_y S_\alpha \subset \tau$\, 
 \item  (\text{\rm Condition B)} $\ell \subset \tau$ .
\end{enumerate}
\end{enumerate}

\section{Definition of $L_\infty$ Kuranishi structures}
\label{sec:kuranishi}

Here is our definition of an $L_\infty$ Kuranishi structure for which we
attract readers' attention to its differences from \cite{fukaya-ono:arnold,fooo:book1,taesu}. 
The current definition is the case where the isotropy group is trivial and so the base of a chart
is a smooth manifold instead of an orbifold. The latter is all we need for the main purpose of
the present paper.  We leave further study of the $L_\infty$ Kuranishi structures in the sense
of this appendix for a future work.

\begin{defn}[$L_\infty$ Kuranishi chart]\label{kurdef}
Let $X$ be a compact metrizable space. A \textit{Kuranishi chart} of $X$ is given by a tuple
\begin{equation}\nonumber
\mathcal{U} = (U, E, \sigma, \psi),
\end{equation}
where
\begin{enumerate}
\item $(U, \omega)$ is a smooth manifold with a presymplectic 2-form $\omega \in \Omega^2(U),$
\item $\pi : E \rightarrow U$ is a (finite rank) vector bundle,
\item $\sigma : U \rightarrow E$ is a smooth section,
\item $\psi : \sigma^{-1}(0) \xrightarrow{\simeq} X$ is a homeomorphism onto its image,
\item  $\mathcal U$ is equipped with an $L_\infty[1]$ algebra 
$\left(\CC_U^E, \{\mathfrak l^\omega_k\}_{k = 1}^\infty\right)$.
\end{enumerate}
\end{defn}

\begin{exm}\label{exm:primary-stabilized}
\begin{enumerate}
\item {[Primary Kuranishi structure]} Take 
$$
\CC_U^E: = \left(\bigoplus \Omega^*(\CF,\omega), \{\mathfrak l^\omega_k\}_{k = 1}^\infty\right)
$$
given by Theorem \ref{thm:oh-park-intro} (\cite[Theorem 9.4]{oh-park:coisotropic}. We call 
the associated Kuranishi structure the \emph{ primary Kuranishi structure}.
\item {[Presymplectic stabilization]} For the construction of virtual fundamental class associated to
a Kuranishi spaces or for the integration of differential forms such as induced closed two
form on the Gromov-Witten moduli space, we need to thicken the primary structure
with virtual dimensional unchanged. This can be done by thickening the original presymplectic manifold
$(U, \omega)$ to $(U \times \R^k, \omega \oplus 0)$. Then a given polarization
$\Gamma_U:  \, TU = G_U\oplus T\CF_U$ canonically thickens to
$$
(\Gamma \oplus \R^k)_{U\times \R^k} : \quad T(U \times \R^k) = (G_U \times \R^k) \oplus T(\CF_U \times \R^k)
$$
as follows: Denote  $\pi_i$, $i=1,\,2$, the projection of $U \times \R^k$ to the $i$-th factor. Then
we have $\pi_1^*\omega = \omega \oplus 0$ and hence
$$
\ker \pi_1^*\omega = \ker( \omega \oplus 0)_{(x,\vec a)} = (\ker \omega)_x \oplus \R^k.
$$
On the other hand, we have
\beastar
T_{(x,\vec a)}(U \times \R^k) & = & TU \times T\R^k|_{(x,\vec a)}  \\
& = & \pi_1^*(G_U \oplus T\CF_U)|_{(x,\vec a)} \oplus \pi_2^*T\R^k |_{(x,\vec a)} \\
& = & G_U|_x \oplus \R^k.
\eeastar
In particular, we have virtual dimension given by
$$
(\dim U + k) - (\rank G + k) = \dim U - \rank G = \text{\rm vit.dim} (U,\omega),
$$
which is unchanged.
\end{enumerate}
 We call this thickening 
 $$
 (U \times \R^k, \pi_1^*\omega, G_U \oplus \R^k)
 $$
 (resp. $(U\times \R^k, \pi_1^*\omega)$)
 the \emph{presymplectic thickening} of $(U, \omega, G_U)$
 (resp. $(U,\omega)$). By construction,
 this thickening does not change the virtual dimension of the Kuranishi structure.
\end{exm}

\subsection{Coordinate change of Kuranishi charts}

Let $\mathcal{U} = (U, E, s, \psi)$ and $\mathcal{U}' = (U', E', s', \psi')$ be Kuranishi
charts of topological spaces $X$ and $Y,$ respectively. Suppose that we are given a 
continuous map $f : X \rightarrow Y.$

\begin{defn}[Coordinate change of $L_\infty$ Kuranishi charts]\label{defn:ourchrtmor}
Let $\mathcal{U} = (U, E, s, \psi)$ and $\mathcal{U}' = (U', E', s', \psi')$
be Kuranishi charts of $X.$
 A \textit{morphism of Kuranishi charts} $\Phi : \mathcal{U} \rightarrow \mathcal{U}'$ is defined by a pair $\Phi = (\phi, \widehat{\phi})$, where
\begin{enumerate}
\item $\phi : U \rightarrow U'$ is a good map of smooth manifolds, which is presymplectic, i.e., $\phi^*\omega' = \omega$,
\item $\psi' \circ \phi = f \circ \psi$ on $\sigma^{-1}(0).,$
\item $\widehat{\phi} : \left(\CC_U^E, \{\mathfrak l^\omega_k\}_{k = 1}^\infty\right) \to 
\phi^*\left(\CC_{U'}^{E'}, \{\mathfrak l^{\omega'}_k\}_{k = 1}^\infty\right)$
is an $L_{\infty}$-morphism.
\end{enumerate}
\end{defn}
 We say  $\Phi = (\phi, \widehat{\phi}) : \mathcal{U} \rightarrow \mathcal{U}'$ is an \emph{open embedding} if $\dim \mathcal{U} = \dim \mathcal{U}'$.

\emph{We would like to attract
readers' attention to the explicit requirement that coordinate change map is
presymplectic} unlike the case of \cite{taesu}.  One may regard that this presymplectic requirement
is closer to the story of `spaces' as that of 
the Kuranishi structure of \cite{fukaya-ono:arnold,fooo:book-kuranishi} than the level of `bundles'.
(Compare this with the definitions in \cite{costello}, \cite{grady-gwiliam1,grady-gwiliam2},
 \cite{pingxu} and \cite{taesu}.)

\begin{defn}[$L_\infty$ Kuranishi spaces]
\label{kstr}
Let $X$ be a compact metrizable space. We say X is a \textit{Kuranishi space} if for each $p \in X$ there exists a neighborhood $V_p$ of $p$ in $X,$ an $L_\infty$ Kuranishi chart $(U_p, E_p, s_p, \psi_p, \omega_p)$  and a homeomorphism $\psi_p : s_p^{-1}(0) \simeq V_p,$ and if $V_p \cap V_p \neq \emptyset,$ we require that there exist an open subchart $\mathcal{U}_{pq}$ of $\mathcal{U}_p$ and an chart embedding $\Phi_{pq} = (\phi_{pq}, \widehat{\phi}_{pq}) : \mathcal{U}_{pq} \hookrightarrow \mathcal{U}_{q}$ over $id_X : X \rightarrow X,$ called \textit{coordinate changes} with the following properties
\begin{enumerate}
\item $\Phi_{pp} = id_{\mathcal{U}_p},$
\item $\psi_q \circ \phi_{pq} = \psi_p$, $\phi_{pq}^*\omega_p = \omega_q$
on $s_{p}^{-1}(0) \cap U_{pq},$
\item $\phi_{qr} \circ \phi_{pq} = \phi_{pr}$ on $\phi_{qr}^{-1}(U_{pq}) \cap U_{pr},$
\item $\psi_p(s_p^{-1}(0) \cap U_{pq}) =\Im (\psi_p) \cap\Im (\psi_q),$
\item $\widehat{\phi}_{qp} : \left(\CC_{U_{pq}}^{E_q}, \{\mathfrak l^{\omega_q}_k\}_{k = 1}^\infty\right) \to 
\phi_{qp}^*\left(\CC_{U_p}^{E_p}, \{\mathfrak l^{\omega_p}_k\}_{k = 1}^\infty\right)$
is an $L_{\infty}$-morphism.
\end{enumerate}
In this situation, we call $\widehat{\mathcal{U}} = (\{\mathcal{U}_p\}, \{\Phi_{pq}\} )$ a \textit{Kuranishi structure} of $X$ and $\{\Phi_{pq}\}$ its \textit{coordinate changes}.
\end{defn}

\subsection{Tangent bundle condition}

Up until now, we have not involved any smooth structure in our discussion of Kuranishi structures.
To perform some differential calculus such as the one to define the de Rham version of
virtual fundamental chains as in \cite{fooo:book-kuranishi},
 we need to involve the notion of `tangent bundle'.
The following is the $L_\infty$ counterpart of FOOO's notion of tangent bundle \cite{fukaya-ono:arnold,
fooo:book1,fooo:book-kuranishi}. 

\begin{defn}[Compare with Definition 3.2 \cite{fooo:book-kuranishi}]\label{defn:Kchart}
Let $\mathcal U = (U,E, \psi,s)$, $\mathcal U'
= (U',E',\psi',s')$ be $L_\infty$ Kuranishi charts of $X$. We say that an $L_\infty$ Kuranishi spaces
$(X, \widehat \CU)$ \emph{has a tangent bundle} if its coordinate change is an
  {\it embedding} of $L_\infty$ Kuranishi charts 
  $: \mathcal U\to \mathcal U'$ is  a pair $\Phi = (\varphi,\widehat\varphi)$ with the following properties
  in addition to those of Definition \ref{defn:ourchrtmor}: 
\begin{enumerate}
\item [{(6)}]  At each zero point $x \in \sigma^{-1}(0)$ the stalk
$(\CC_U^E, \{\mathfrak l^\omega_k\}_{k = 1}^\infty)|_x$ extends to a germ of curved $L_\infty[1]$ algebras
such that 
$$
\left(\Gamma(TU) \stackrel{\sigma_*}{\longrightarrow} (\CC_U^E)_1 \stackrel{\mu_1}{\longrightarrow}
 (\CC_U^E)_2 \stackrel{\mu_1}{\longrightarrow}  (\CC_U^E)_3 \stackrel{\mu_1}{\longrightarrow}
 \cdots \right)\Big|_x
 $$
such that $\mathfrak l^\omega_0 = \sigma_*: = \nabla \sigma: \Gamma(TU) \to \Gamma(E)$ and 
satisfies the following:
\item [{(7)}] At $x \in \sigma^{-1}(0)$, the derivative
$\nabla_{\varphi(x)}\sigma' $ induces an isomorphism
\be \label{eq:tangent-bundle}
\frac{T_{\varphi(x)}U'}{(D_x\varphi)(T_xU)}
\cong
\frac{E'_{\varphi(x)}}{\widehat\varphi(E_x)}
\ee
such that  the map \eqref{eq:tangent-bundle} is 
the right vertical arrow of the next commutative diagram denoted by ${\mathfrak f_1}$
\be\label{eq:tangent-diagram}
\xymatrix{
\Gamma(T_xU) \ar[r]^ {D_x\varphi}  \ar[d]^{D_xs} & T_{\varphi(x)}U' \ar[r] \ar[d]^{_{\varphi(x)}s'} &
\frac{T_{\varphi(x)}U'}{(D_x\varphi)(T_xU)}  \ar[d]^{\mathfrak f_1} \\
{\widehat\varphi}(E_x) \ar[r]^{\widehat{\varphi}}& E'_{\varphi(x)}
\ar[r] & \frac{E'_{\varphi(x)}}{{\widehat \varphi}(E_x)}
}
\ee
Then $\mathfrak f_1$ and the above diagram is promoted to a diagram of curved
$L_\infty$ morphisms (in the level of germs at $x$).
\end{enumerate}
\end{defn}

\section{Special connection associated to a closed 2-form}

In this section, we prove a result of a special connection adapted to any closed
2-form which itself has some independent interest on its own.

\begin{thm}\label{thm:omega-special} 
Let $\omega$ be any closed form and let $x \in M$ be given.
Then there exists an affine connection
$ \nabla$ that satisfies $ \nabla \omega|_x = 0$.
\end{thm}
\begin{proof}  
Denote by $\nabla^g$ the Levi-Civita connection of $g$, and define another connection
$\nabla$ by putting
$$
\nabla = \nabla^g + B
$$
for some $(1,1)$ tensor $B$. We will solve the equation
\be\label{eq:tosolve}
0 = \nabla \omega|_x = \nabla^g \omega|_x + B(x)\omega|_x
\ee
for $B$ on a neighborhood $U$ of the given point $x \in M$. 

Let $(x_1,\cdots, x_n)$ be a geodesic normal  coordinates of $g$ such that 
$$
\ker \omega_x = \span \{\del_1, \cdots, \del_m\}.
$$
We write by $\Gamma_{ij}^k$ the associated Christoeffl symbols. Then 
we have $\Gamma_{ij}^k(x) = 0$ by definition of normal coordinates.
We can refine the property of this geodesic normal coordinates incorporating
the presence of a 2-form $\omega$ as follows.
We can express 
$$
\omega|_x = \sum_{i, j} \omega_{ij}(x) dx_i|_x \wedge dx_j|_x = \sum_{i, j > m} \omega_{ij}(x) dx_i|_x \wedge dx_j|_x.
$$
Therefore since $\Gamma_{ij}^k(x) = 0$, the equation 
$$
\nabla _{\del_k}\omega|_x = 0 \quad \text{\rm for all }\, k = 1, \ldots, N
$$
is reduced to
$$
0 = \sum_{i,j > m} \omega_{ij,k}(x) dx_i \wedge dx_j + \sum_{i, j} B_{ij,k}(x) dx_i \wedge dx_j \quad \forall 
k = 1, \ldots, N.
$$
Therefore we have only to choose $B_{ij}$ on a neighborhood $V$ with $\overline V \subset U$ so that
on a neighborhood of $\overline V$ contained in $U$ so that
\be\label{eq:Bij}
B_{ij}(x) = \begin{cases} - \omega_{ij}(x) \quad & \text{\rm for }\, i, \, j > m \\
0 \quad & \text{\rm otherwise}.
\end{cases}
\ee
We denote by $\nabla^{U}$ the locally defined connection.
Then we apply the partitions of unity $\{\chi, 1- \chi\}$ for $\{U, M \setminus \overline V\}$
to glue $\nabla^U$ on $U$ and $\nabla^g$ on $M \setminus \overline V$ to define
$$
\nabla = \chi \nabla^U + (1- \chi) \nabla^g.
$$
This then satisfies the requirement at $x \in M$, which finishes the proof.
\end{proof}

\section{Proof of Lemma \ref{lem:eta=0}}
\label{sec:eta=0}

Let $(Q,\eta)$ lie in the cokernel of $d\Upsilon(x,\omega)$ with $d\omega = 0$. To avoid notational confusion
in the proof below, we put $x_0$ the given point $x$ in the text of Lemma \ref{lem:eta=0} and reserve  $x$ a general point in $M$.  We also write $Q= Q_{x_0}$ to emphasize the fact $Q \in T_{x_0} \Lambda_m(M)$.
By an abuse of notation, we also sometimes write $(x^1,\cdots, x^n) = x$ in coordinates which should not confuse readers.

Noting from \eqref{eq:zeta} that $\Delta \zeta = 0$ on $M \setminus \{x_0\}$, we derive
$\zeta$ is smooth on $M \setminus \{x_0\}$ by the local elliptic regularity. This implies 
$\zeta$ has a point singularity at $x_0$. Therefore it is enough to examine the behavior of $\zeta$
in a sufficiently small coordinate neighborhood.

We take the geodesic normal  coordinates $(x^1, \cdots, x^N)$ of a neighborhood $U$ 
centered at the given point $x_0$ so that
\be\label{eq:g}
g = \sum_{i=1} dx^i dx^i + O(r), \quad r^2 = (x^1)^2 + \cdots + (x^n)^2.
\ee
Recall that the associated Green function
$G = G(x,0)$ has singularity at $x = 0$ of the type 
\be\label{eq:singularity}
C r^{2-N} + f(x)
\ee
in some local coordinates $r = \sqrt{x_1^2 + \cdots + x_N^2}$, \emph{when $N \geq 3$
which we are currently assuming}. Here $f$ is some smooth function with $f(x_0)=0$ so that $f = O(r)$
Since $f = O(r)$ will not affect the proof below, we will 
drop it from now on which is harmless and whose effect will be subsumed in $O(r)$ below.
Therefore $\eta = d\zeta$ (for $\zeta = G Q^\perp$) has singularity at $x_0 = 0$ is of the type $C r^{1-N}$.
(See \cite[p. 115, Lemma 1]{deRham}, \cite[Equation (2.12)]{gilbarg-trudinger} for the details.)
It follows from this that the function $r^{1-N}$  is contained in $L^q_{loc}$ if and only if
$$
N-1 - q(N-1) > -1.
$$
This  inequality is automatic since we have already chosen $p > N$ above, which can be checked 
equivalent to $p > N$.
Therefore $\zeta$ is a distributional solution to \eqref{eq:zeta} lying in $L^q_{loc}$.

By the above discussion, we have
 $$
 \zeta =  G\, Q^\perp_{x_0}
 $$
on a neighborhood $U$ of $p$ where $G = \frac1{r^{N-2}}$ is the Green function (modulo multiplication of
a nonzero positive constant) written in 
the geodesic normal coordinate $(x^1, \cdots, x^N)$ centered at $p$ with $r^2$.  
Therefore to prove the lemma, it will be enough to prove $Q_{x_0}^\perp = 0$, 
which we will prove by contradiction.

\emph{Since we have chosen $\zeta$ so that $\delta \zeta = 0$ weakly}, we have
\be\label{eq:intzetadbeta=0}
 \int \langle \zeta, d\beta \rangle \, d\vol_g = 0
\ee
 for all smooth one-form $\beta$.  For $\zeta$ of the form above, this equation becomes
 $$
  \int_U \frac{\langle Q_{x_0}^\perp, d\beta \rangle}{r^{N-2}} \, d\vol_g = 0
 $$
for all $\beta$ supported in $U$, when $\zeta =   G\, Q^\perp_{x_0}$.

Now suppose to the contrary that 
$$
Q^\perp_{x_0} \neq 0.
$$
 We can rewrite the integral into
\beastar
 \int_U \frac{\langle Q_{x_0}^\perp, d\beta \rangle}{r^{N-2}} \, d\vol_g
 & = & \int_U \frac{\langle Q_{x_0}^\perp, d \beta \rangle}{r^{N-2}} \, r^{N-1}\,  dr d\Theta_g \\
& = & \int_U  \langle Q_{x_0}^\perp, d\beta \rangle\,r\,  dr \, d\Theta_g. 
\eeastar
Here $d\Theta_g$ denotes the spherical volume form on the unit sphere.

In the remaining proof, we will find a one-form $\beta$ of the above form such that
\be\label{eq:int>0}
 \int_U  \langle Q_{x_0}^\perp, d\beta \rangle\,r\,  dr \, d\Theta_g > 0
 \ee
We express
\be\label{eq:Qperp}
Q_{x_0}^\perp = \sum_{i \, \text{\rm or}\, j > 2k} b_{ij} dx^i \wedge dx^j
\ee
for some $b_{ij}$ satisfying $b_{ji} = - b_{ij}$, in terms of the
geodesic normal coordinates  such that  
$$
T_0 (\Lambda_m^2(M)) = \span \{dx^i \wedge dx^j\}_{1 \leq i, \, j \leq 2k}, \quad m = N - 2k
$$
and 
$$
N_0 (\Lambda_m^2(M)) = \left(T_0 (\Lambda_m^2(M)) \right)^\perp, \quad 0 = d\omega
$$
Then we have
$Q_{x_0}^\parallel = \sum_{i,j \leq 2k} b_{ij} dx^i|_x \wedge dx^j|_{x_0}$.  Since we assume $Q_{x_0}\perp \neq 0$,
there exists a pair $(i_0,j_0)$ such that
\be\label{eq:bij0}
b_{i_0j_0} \neq 0.
\ee
We choose $\beta$ of the form
\be\label{eq:beta}
\beta = \chi(r)\,  x_{i_0} b_{i_0j_0} dx_{j_0} \quad \text{\it `no summation involved'}
\ee
for a cut-off function 
$\chi: U \to [0,1]$ so that 
$\supp \chi \subset U$ and
$$
\chi(x) \equiv 1 \quad \text{\rm for }\, x \in V \subset \overline V \subset U.
$$
Then we have
\beastar
d\beta & = & \chi(r) b_{i_0j_0}\, dx^{i_0} \wedge dx^{j_0} + \chi'(r)b_{i_0j_0}\, dr \wedge dx^{j_0}\\
& = & \chi(r) b_{i_0j_0} \, dx^{i_0} \wedge dx^{j_0} + \frac{\chi'(r)}{r} \sum_j b_{i_0j_0}\, x^j 
dx^j \wedge dx^{j_0}.
\eeastar
Note 
$$
\supp d\chi \subset V_2 \setminus \overline V_1
$$
 where $V_1: = V$ and
$$
x \in V_1 \subset \overline V_1 \subset V_2 \subset \overline V_2 \subset U.
$$
We take $V_1 = B_g(\delta_1)$ and $V_2 = B_g(\delta_2)$ with $0 < \delta_1 < \delta_2 < \iota_g$
where $B_g(\delta)$ is the geodesic ball of radius $\delta < \iota_g$,
 the injectivity radius $\iota_g$ of the given metric $g$.
Therefore, using the identity 
$$
\langle dx^j \wedge dx^{j_0}, dx^{i_0}\wedge dx^{j_0} \rangle = \delta_{ji_0} + O(|x|)
$$
which follows from \eqref{eq:g}, we derive
\beastar
\langle Q_{x_0}^\perp, d\beta \rangle(x) & = & \chi(r)  b_{i_0j_0}^2 + \frac{\chi'(r)}{r} \sum_j x^j 
\delta_{ji_o} b_{i_0j_0}^2 + O(|x|)b_{i_0j_0}\\
& = & \chi(r)  b_{i_0j_0}^2 + \frac{\chi'(r)}{r}  x^{i_0} 
b_{i_0j_0}^2 + O(|x|)b_{i_0j_0} .
\eeastar
We then have
\beastar
\int_U  \langle Q^\perp_{x_0}, d\beta \rangle\,r\, dr \, d\Theta_g & = & 
\int_{B_g(\delta_2) }\chi(r)  b_{i_0j_0}^2 \, r\, dr \, d\Theta_g
+ \int_{B_g(\delta_2)}   \frac{\chi'(r)}{r} \ x^{i_0} b_{i_0j_0}^2\, r\, dr\, d\Theta_g \\
&{}& + \int_{B_2(\delta_2)} O(\delta_2)  r\, dr\, d\Theta_g.
\eeastar

Then we have the integral
\be\label{eq:1st-integral}
\int_U \chi(r)  b_{i_0j_0} ^2\, r\, dr \, d\Theta_g \geq \int_{B_g(\delta_1)} b_{i_0j_0}^2 \, r\, dr \, d\Theta_g
\geq \frac12 b_{i_0j_0}^2 \delta_1^2 > 0
\ee
On the other hand, the integral of the middle summand
against $d\vol_g = r^{N-1}dr\, d\Theta_g$ vanishes 
by the skew-symmetry of the coordinate function $x^{i_0}$.

Finally, obviously we have
$$
\int_{B_2(\delta_2)} O(\delta_2)  r\, dr\, d\Theta_g \sim O(\delta_2^3).
$$
Therefore we obtain
\beastar
\int_U  \langle Q_{x_0}^\perp, d\beta \rangle\,r\, dr \, d\Theta_g \geq \frac12 b_{i_0j_0}^2 \delta_1^2
- C\delta_2^3
\eeastar
for some constant $C> 0$ depending only on $g$ and the geodesic normal coordinates.
Therefore if we choose $0 < \delta_1 < \delta_2 < \iota_g$ so that
$$
 \frac12 b_{i_0j_0}^2 \delta_1^2 - C\delta_2^3 > 0,
$$
i.e., if we choose them so that 
\be\label{eq:key-inequality}
\delta_1 < \delta_2, \quad
\left(\frac{\delta_1}{\delta_2}\right)^2 > \left(\frac{2C}{b_{i_0j_0}^2}\right)  \delta_2
\ee
we have
$$
\int_U  \langle Q_{x_0}^\perp, d\beta \rangle\,r\, dr \, d\Theta_g > 0.
$$
Note that the second inequality of \eqref{eq:key-inequality}
 is obvisouly possible by taking $\delta_2> \delta_1$
as close as possible, while letting $\delta_2$ sufficiently small. For example, we may choose
$$
\delta_2 = \min\left\{\frac12, \frac{b^2_{i_0j_0}}{4C} \right\} , \quad  \delta_1= \frac34 \delta_2
$$
which will do the purpose. This non-vanishing contradicts to the vanishing \eqref{eq:intzetadbeta=0},
which finishes the proof of the lemma.

\def\cprime{$'$}
\providecommand{\bysame}{\leavevmode\hbox to3em{\hrulefill}\thinspace}
\providecommand{\MR}{\relax\ifhmode\unskip\space\fi MR }
\providecommand{\MRhref}[2]{%
  \href{http://www.ams.org/mathscinet-getitem?mr=#1}{#2}
}
\providecommand{\href}[2]{#2}



\end{document}